\documentclass[11pt]{amsart}

\usepackage{amsmath, amssymb, amsthm, amsfonts, amsxtra, amscd}

\input xy
\xyoption{all}

\swapnumbers
\numberwithin{equation}{section}

\theoremstyle{plain}
\newtheorem{theorem}[subsection]{Theorem}
\newtheorem{lemma}[subsection]{Lemma}

\theoremstyle{definition}

\newtheorem{remark}[subsection]{Remark}

\def\CC{\mathbb{C}}

\def\FF{\mathbb{F}}
\def\GG{\mathbb{G}}

\def\NN{\mathbb{N}}

\def\QQ{\mathbb{Q}}

\def\ZZ{\mathbb{Z}}

\def\calA{\mathcal{A}}

\def\calC{\mathcal{C}}
\def\calD{\mathcal{D}}

\def\calH{\mathcal{H}}

\def\calO{\mathcal{O}}

\def\calS{\mathcal{S}}

\def\calV{\mathcal{V}}

\def\bC{\mathbf{C}}

\def\bI{\mathbf{I}}

\def\bK{\mathbf{K}}

\def\bP{\mathbf{P}}

\def\bS{\mathbf{S}}

\newcommand\frt{\mathfrak{t}}

\newcommand\tilW{\widetilde{W}}
\newcommand\tilLam{\widetilde{\Lambda}}

\newcommand\tg{\widetilde{g}}

\newcommand{\Gr}{\textup{Gr}}
\newcommand{\gr}{\textup{gr}}
\newcommand\id{\textup{id}}
\newcommand{\IH}{\textup{IH}}

\newcommand\Out{\textup{Out}}

\newcommand\pt{\textup{pt}}
\newcommand\pr{\textup{pr}}

\newcommand\Rep{\textup{Rep}}

\newcommand\Sym{\textup{Sym}}

\newcommand{\tr}{\textup{tr}}

\newcommand{\Vect}{\textup{Vec}}

\newcommand\Aut{\textup{Aut}}

\newcommand{\Gm}{\GG_m}

\newcommand{\ad}{\textup{ad}}
\newcommand{\Ad}{\textup{Ad}}
\renewcommand\sc{\textup{sc}}

\newcommand\xch{\mathbb{X}^*}
\newcommand\xcoch{\mathbb{X}_*}

\newcommand{\isom}{\stackrel{\sim}{\to}}
\newcommand{\leftexp}[2]{{\vphantom{#2}}^{#1}{#2}}

\newcommand{\twtimes}[1]{\stackrel{#1}{\times}}
\newcommand{\cohog}[2]{\textup{H}^{#1}({#2})}     % plain group
\newcommand{\cohoc}[2]{\textup{H}_{c}^{#1}({#2})}     % compact support
\newcommand{\jiao}[1]{\langle{#1}\rangle}
\newcommand{\wt}[1]{\widetilde{#1}}

\newcommand{\Fl}{\textup{Fl}}
\newcommand{\dotw}{\dot{w}}
\newcommand{\doty}{\dot{y}}
\newcommand{\dota}{\dot{a}}
\newcommand{\upH}{\textup{H}}
\newcommand{\Hb}{\upH^{\bullet}}
\newcommand{\Hbs}{\upH^{\bullet,\sigma}}
\newcommand{\IHb}{\IH^{\bullet}}
\newcommand{\dG}{\check{G}}
\newcommand{\dT}{\check{T}}
\newcommand{\dB}{\check{B}}

\newcommand{\pol}{\textup{pol}}
\newcommand{\sph}{\textup{sph}}
\newcommand{\Tad}{T^{\ad}}
\newcommand{\Gad}{G^{\ad}}
\newcommand{\Grad}{\Gr^{\ad}}
\newcommand{\dGsc}{\dG^{\sc}}
\newcommand{\dTsc}{\dT^{\sc}}
\newcommand{\barA}{\overline{A}}
\newcommand{\bary}{\overline{y}}
\newcommand{\barw}{\overline{w}}
\newcommand{\barW}{\overline{W}}
\newcommand{\calSad}{\calS^{\ad}}

\newcommand{\ep}{\epsilon}
\newcommand\s{\sigma}
\renewcommand\t{\tau}
\renewcommand\l{\lambda}
\renewcommand\L{\Lambda}
\newcommand\tZ{\wt{Z}}
\newcommand{\quash}[1]{}

\title{A $(-q)$-analogue of weight multiplicities}
\author{George Lusztig}
\thanks{G.L. is partially supported by the NSF grant DMS-0758262.}
\address{Department of Mathematics, MIT, 77 Massachusetts Avenue, Cambridge, MA 02139}
\email{gyuri@math.mit.edu}

\author{Zhiwei Yun}
\thanks{Z.Y. is partially supported by the NSF grant DMS-0969470.}
\address{Department of Mathematics, MIT, 77 Massachusetts Avenue, Cambridge, MA 02139}
\email{zyun@math.mit.edu}
\date{}

\begin{document}

\begin{abstract}
We prove a conjecture in \cite{L} stating that certain polynomials $P^{\sigma}_{y,w}(q)$ introduced in \cite{LV1} for twisted involutions in an affine Weyl group give $(-q)$-analogues of weight multiplicities of the Langlands dual group $\dG$. We also prove that the signature of a naturally defined hermitian form on each irreducible representation of $\dG$ can be expressed in terms of these polynomials $P^{\s}_{y,w}(q)$.
\end{abstract}

\maketitle

%%% intro %%%%

\section{Statement of the main theorems}

\subsection{The $P^{\sigma}$-polynomials}\label{ss:Ps} Let $W$ be a Coxeter group with simple reflections $S$. Let $\ell:W\to\NN$ be the length function defined by the simple reflections $S$. In \cite{KL}, for any two elements $y,w\in W$, a polynomial $P_{y,w}(q)\in\ZZ[q]$ is attached. Consider the Hecke algebra $\calH$ over $\calA=\ZZ[q,q^{-1}]$ ($q$ is an indeterminate) with basis $\{T_w\}_{w\in W}$ and multiplication given by $T_wT_{w'}=T_{ww'}$ if $\ell(ww')=\ell(w)+\ell(w')$ and $(T_s+1)(T_s-q)=0$ for all $s\in S$. Then $\{\sum_{y\in W;y\leq w}P_{y,w}(q)T_y\}_{w\in W}$ is (up to a factor) the ``new basis''  of $\calH$ introduced in \cite{KL}.

In \cite{LV1} (for $W$ a Weyl group) and \cite{L} (in general), the authors work in the situation of a triple $(W,S,*)$ where $(W,S)$ is as before and $*$ is an involution of $(W,S)$. Let $I_{*}=\{w\in W|w^{*}=w^{-1}\}$ be the $*$-twisted involutions in $W$. From the data $(W,S,*)$, a refined version $P^{\sigma}_{y,w}(q)\in\ZZ[q]$ of $P_{y,w}(q)$ is defined for $y,w\in I_{*}$. They also introduced a free $\calA$-module $M$ with basis $\{a_w\}_{w\in I_*}$, which carries a natural module structure over the Hecke algebra $\calH'$ with $q$ replaced by $q^2$. Then $\{\sum_{y\leq w,y\in I_{*}}P^{\s}_{y,w}(q)a_{y}\}_{w\in I_{*}}$ is (up to a factor) the ``new basis'' of $M$ introduced in \cite[Theorem 0.3]{LV1} and \cite[Theorem 0.4]{L}.

\subsection{Affine Weyl group}\label{ss:aff}
For the rest of the note we consider the setting of \cite[Section 6]{L}: $(W,S)$ is the Coxeter group associated to an untwisted connected affine Dynkin diagram. Let $\Lambda\subset W$ be the subgroup of {\em translations}, i.e., those elements which have finite conjugacy classes. This is a free abelian subgroup of $W$ of finite index. Let $\barW=W/\Lambda$. We shall use additive notation for the group law in $\Lambda$. The conjugation action of $w\in \barW$ on $\Lambda$ is denoted by $\lambda\mapsto\leftexp{w}{\lambda}$.

Fix a hyperspecial vertex $s_{0}\in S$ (i.e., a vertex in $S$ with Dynkin label equal to 1). Then the finite Weyl group $W_{J}$ generated by $J=S-\{s\}$ is a section of the natural projection $W\to \barW$, and we henceforth identify $W_{J}$ with $\barW$. Let $w_{J}$ be the longest element of $W_{J}$.

An element $\l\in\L$ is {\em dominant} if $\ell(\l w_{J})=\ell(\l)+\ell(w_{J})$. Let $\L^{+}$ denote the set of dominant translations. The set of double cosets $W_{J}\backslash W/W_{J}$ is in bijection with $\L^{+}$: each $W_{J}$-double coset in $W$ contains a unique $\l\in\L^{+}$. For $\lambda\in\L^{+}$, let $d_{\lambda}=\l w_{J}$ be the longest element in the double coset $W_{J}\lambda W_{J}$.

Let $*$ be the automorphism of $W$ defined by
\begin{eqnarray}\label{star}
w^{*}&:=&w_{J}ww_{J}, \textup{ for }w\in W_{J};\\
\notag\lambda^{*}&:=&-\leftexp{w_{J}}{\lambda} \textup{ for }\lambda\in \Lambda.
\end{eqnarray}
This $*$ is an involution which stabilizes $S$ and fixes $s_{0}$. In fact, if $w_{J}$ acts by $-1$ on $\Lambda$, then $*$ is the identity; otherwise $*$ has order two. As shown in \cite[Proposition 8.2]{L}, every element $d_{\lambda}$ belongs to $I_{*}$. Therefore we may consider the polynomials $P^{\sigma}_{d_{\mu},d_{\lambda}}(q)$.

The following theorem is the main result of this note, which was conjectured by the first author in \cite[Conjecture 6.4]{L}.
\begin{theorem}\label{th:main} Notation as above. Then for any $\lambda,\mu\in\L^{+}$, we have 
\begin{equation*}
P^{\sigma}_{d_{\mu},d_{\lambda}}(q)=P_{d_{\mu},d_{\lambda}}(-q).
\end{equation*}
\end{theorem}
The proof of the theorem will be given in Section \ref{s:proof}, after some preparation regarding the geometric Satake equivalence in Section \ref{s:geom}. In Section \ref{s:gen}, we give a generalization of the above theorem to other involutions $\diamond$ of $(W,S)$ which are closely related to $*$.

In \cite[Proposition 8.6]{L}, the first author proves special cases of this result by pure algebra.

It is proved in \cite[6.1]{L83} that $P_{d_{\mu},d_{\lambda}}(q)$ is a $q$-analogue of the $\mu$-weight multiplicity in the irreducible representation $V_{\lambda}$ of an algebraic group $\dG$ (see the discussion in Section \ref{ss:Gr}). Therefore, we may interpret the above theorem as saying that $P^{\sigma}_{d_{\mu},d_{\lambda}}(q)$ is a $(-q)$-analogue of weight multiplicities, hence the title of this note.

\subsection{The $Z^{\s}$-polynomials}\label{ss:Zs} The polynomials $P_{y,w}(q)$ is the Poincar\'e polynomial of the local intersection cohomology of an affine Schubert variety indexed by $w$; the Poincar\'e polynomial of the global intersection cohomology of the same affine Schubert variety is given by
\begin{equation*}
Z_w(q)=\sum_{y\in W;y\leq w}P_{y,w}(q)q^{\ell(y)}\in\ZZ[q].
\end{equation*}
Algebraically, consider the $\calA$-algebra homomorphism $\chi:\calH\to\calA$ given by $\chi(T_w)=q^{\ell(w)}$ for all $w\in W$. Then $Z_{w}(q)$ is the value of the new basis $\sum_{y\leq w}P_{y,w}(q)T_{y}$ under the homomorphism $\chi$.

We want to define some polynomials $Z_w^\s(q)\in\QQ(q)$ which play the same role with respect to $Z_w(q)$ as $P_{y,w}^\s(q)$ plays with respect to $P_{y,w}(q)$. To to do, we replace $\chi:\calH\to\calA$ by the following $\calA$-linear map introduced in \cite[5.7]{L}
\begin{eqnarray}\label{zeta}
\zeta:M&\to&\QQ(q)\\
a_w&\mapsto&q^{\ell(w)}\left(\frac{q-1}{q+1}\right)^{\phi(w)} \textup{ for all } w\in I_*
\end{eqnarray}
Here $\phi:I_*\to\NN$ is defined in \cite[4.5]{L}. Concretely, for $w\in I_{*}$ with image $\barw\in\barW$, $\phi(w)=e(\barw*)-e(*)$, where $e(*)$ (resp. $e(\barw*)$) is the dimension of the $(-1)$-eigenspace of the involution $t\mapsto t^{*}$ (resp. $t\mapsto w(t^{*})$) on $\L_{\QQ}=\L\otimes_{\ZZ}\QQ$. 

For $w\in I_*$ we let $Z^{\s}_{w}(q)$ be the image of the new basis of $M$ under $\zeta$:
\begin{equation}\label{defineZ}
Z^\s_w(q)=\zeta\left(\sum_{y\in I_*;y\leq w}P_{y,w}^\s(q) a_y\right)=\sum_{y\in I_*;y\leq w}P_{y,w}^\s(q) q^{\ell(y)}\left(\frac{q-1}{q+1}\right)^{\phi(y)}\in\QQ(q)
\end{equation}
We also set
\begin{equation}\label{definetZ}
\tZ_{d_\l}(q)=Z_{d_\l}(q)Z_{w_{J}}(q)^{-1}\in\QQ(q), \hspace{.5cm} \tZ^\s_{d_\l}(q)=Z^\s_{d_\l}(q)Z^\s_{w_{J}}(q)^{-1}\in\QQ(q).
\end{equation}
Our second main result is
\begin{theorem}\label{th:Z-q} For any $\l\in\L^+$ we have $\tZ^\s_{d_\l}(q)=\tZ_{d_\l}(-q)$. In particular, $\tZ^\s_{d_\l}(q)\in\ZZ[q]$.
\end{theorem}
We will present two proofs of the theorem, one geometric in Section \ref{s:g2} which is based on a cohomological interpretation of $Z^{\s}_{w}(q)$, and one algebraic in Section \ref{s:a2}. Both proofs rely on Theorem \ref{th:main}.

It is also observed in \cite{L83} that $\tZ_{d_{\l}}(q)$ is a $q$-analogue of the dimension of the irreducible representation $V_{\l}$ of the group $\dG$. We will show in Section \ref{ss:sign} that $\tZ^{\s}_{d_{\l}}(q)$ is a $q$-analogue of the signature of $V_{\l}$ under a naturally defined hermitian form introduced in \cite{L97}.

\subsection{Gelfand's trick} It is interesting to notice the relation between the involution $*$ and ``Gelfand's trick'' in proving that the spherical Hecke algebra is commutative. In fact, for a split simply-connected almost simple group $G$ over a local field $F$ with Weyl group $W_{J}$, the double coset $G(\calO_{F})\backslash G(F)/G(\calO_{F})$ is in bijection with $W_{J}\backslash W/W_{J}$. The spherical Hecke algebra $\calH^{\sph}$ consists of compactly supported bi-$G(\calO_F)$-invariant functions on $G(F)$ with the algebra structure given by convolution. There is an involution $g\mapsto g^{*}$ of $G$ which stabilizes a split maximal torus $T$ and acts by $-w_{J}$ on $\xcoch(T)=\Lambda$. The induced action on the affine Weyl group $W$ is the same as the one given in Section \ref{ss:aff}. The anti-involution $\tau:g\mapsto(g^{*})^{-1}$ induces an anti-involution on $\calH^{\sph}$ while fixing each double coset $W_{J}\backslash W/W_{J}$, hence acting by identity on $\calH^{\sph}$. This implies the commutativity of $\calH^{\sph}$. Roughly speaking, the main theorem is a categorification of Gelfand's trick: it explains what $\tau$ does to the Satake category (categorification of $\calH^{\sph}$) beyond the level of isomorphism classes of objects (on which it acts by identity).

%This is the mistake in the construction of the commutativity constraint of the Satake category in \cite[Proposition 2.3.1]{G}.

\subsection{Notation and conventions} By a tensor category, we mean a monoidal category with a commutativity constraint compatible with the associativity constraint.

For an algebraic torus $T$, let $\xcoch(T)$ (resp. $\xch(T)$) denote the group of cocharacters (resp. characters) of $T$. For a cocharacter $\lambda:\Gm\to T$, we use $x^{\lambda}$ to mean the image of $x\in\Gm$ under $\lambda$; for a character $\alpha:T\to\Gm$, we use $z^{\alpha}$ to denote the image of $z\in T$ under $\alpha$. Note that $(x^{\lambda})^{\alpha}=x^{\jiao{\alpha,\lambda}}\in\Gm$.

By an involution in a group, we mean an element of order at most two.

All algebraic varieties in this note are over $\CC$; all complexes of sheaves are with $\QQ$-coefficients.

For an algebraic variety $X$ of dimension $n$, let $\IHb(X)$ denote its intersection cohomology groups with $\QQ$-coefficients. We normalize it so that $\IH^{i}(X)=0$ unless $0\leq i\leq 2n$.

\section{Geometric definition of the $P^{\sigma}$-polynomials}\label{s:geom}
\subsection{Affine flag variety}\label{ss:flag}
In this section we give a geometric definition of the polynomials $P^{\sigma}_{x,y}(q)$. In fact, in the case of  finite Weyl groups with $*=\id$, such a geometric definition is given in \cite[Section 3]{LV1} using the geometry of flag varieties. It is remarked in \cite[Section 7.1-7.2]{LV1} that such a geometric definition works for affine Weyl groups and general $*$, with the flag varieties replaced by affine flag varieties. This section is an elaboration of this remark.
 
Let $G$ be the simply-connected almost simple group over $\CC$ whose extended Dynkin diagram is the one we started with in Section \ref{ss:aff}, so that the usual Dynkin diagram of $G$ is given by removing the vertex $s_{0}$. Fix a pinning for $G$; in particular, fix a maximal torus $T\subset G$, and a Borel $B$ containing $T$. We may identify $(W_{J},S-\{s_{0}\})$ with the Weyl group $N_{G}(T)/T$ together with the simple reflections determined by $B$. We may also identify $\Lambda$ with the cocharacter lattice $\xcoch(T)$, which is also the coroot lattice of $G$.

Let $G((t))$ be the loop group associated to $G$: it is the ind-scheme representing the functor $R\mapsto G(R((t)))$ for any $\CC$-algebra $R$. Let $G[[t]]\subset G((t))$ be the subscheme representing the functor $R\mapsto G(R[[t]])$. The affine Weyl group $W$ may be identified with the $\CC$-points of $N_{G((t))}(T((t)))/T[[t]]$. For each $w\in W$, we choose a lifting $\dotw$ of it in $N_{G((t))}(T((t)))$. For example, if $\lambda\in\Lambda$, we may choose $\dot{\lambda}$ to be the point $t^{\lambda}\in T((t))$.

An Iwahori subgroup of $G((t))$ is one which is conjugate to $\bI=\pi^{-1}(B)\subset G[[t]]$ where $\pi:G[[t]]\to G$ is the mod $t$ reduction morphism. Let $\Fl=G((t))/\bI$ be the affine flag variety of $G$ classifying Iwahori subgroups of the loop group $G((t))$. This is a (locally finite) infinite union of projective varieties over $\CC$ of increasing dimensions. The group scheme $\bI$ acts on $\Fl$ from the left with orbits $\Fl_{w}=\bI\dotw \bI/\bI$ indexed by $w\in W$. Each orbit $\Fl_{w}$ is isomorphic to an affine space of dimension $\ell(w)$ (with respect to the simple reflections $S$). Let $\Fl_{\leq w}$ be the closure of $\Fl_{w}$, which is the union of $\Fl_{y}$ for $y\leq w$.

Consider the derived category $D_{\bI}(\Fl)=\varinjlim_{w\in W} D_{\bI}(\Fl_{\leq w})$ of $I$-equivariant $\QQ$-complexes which are supported on the $\Fl_{\leq w}$ for some $w\in W$.  Note that for fixed $w$, the $I$-action on $\Fl_{\leq w}$ factors through a quotient group scheme $\bI_{w}$ of finite type such that $\ker(\bI\to\bI_{w})$ is pro-unipotent. We therefore understand $D_{\bI}(\Fl_{\leq w})$ as the category of $\bI_{w}$-equivariant derived category of $\QQ$-complexes on the projective variety $\Fl_{\leq w}$ in the sense of \cite{BL}.

\subsection{Geometric interpretation of the $P^{\s}$-polynomials} Let $*$ denote the pinned automorphism of $G$ such that $\lambda\mapsto (\leftexp{w_{J}}{\lambda})^{*}$ acts by $-1$ on $\Lambda$. This involution induces an involution on the affine Weyl group $(W,S)$ which coincides with the $*$ defined in \eqref{star}. The involution $*$ also induces an involution on $G((t))$ preserving the Iwahori $\bI$, so that it induces an involution on $\Fl$ which we still denote by $*$. 

Consider the anti-involution $\tau$ of $G((t))$ defined as
\begin{equation*}
\tau(g)=(g^{*})^{-1}.
\end{equation*}
We would like to define a functor:
\begin{equation*}
\tau^{*}:D_{\bI}(\Fl)\to D_{\bI}(\Fl)
\end{equation*}
given by pull-back along the map $\tau$. We may identify each object of $D_{\bI}(\Fl)$ as a complex on $G((t))$ equivariant under the left and right translation by $\bI$. Since each $\bI$-double coset $\bI\dotw\bI\subset G((t))$ is sent to another double coset $\bI(\dotw^{*})^{-1}\bI$, pull-back by $\tau$ preserves bi-$\bI$-equivariance, and defines the functor $\tau^{*}$.

For each object $\bK\in D_{\bI}(\Fl)$ and $y\in W$, the restriction of $\bK$ to $\Fl_{y}$ is a constant complex by $\bI$-equivariance. We therefore have a vector space $\calH^{i}_{y}\bK$, which is canonically isomorphic to the $i$-th cohomology of the stalk of $\bS_{w}$ at any point of $\Fl_{y}$. 

For each $w\in W$, one has the (shifted) intersection cohomology complex $\bS_{w}\in D_{\bI}(\Fl)$ of $\Fl_{\leq w}$, which we normalize so that $\bS_{w}|_{\Fl_{w}}\cong\QQ$. If $w\in I_{*}$ (i.e., $(w^{*})^{-1}=w$), we have a canonical isomorphism
\begin{equation}\label{invSw}
\Phi_{w}:\tau^{*}\bS_{w}\isom\bS_{w}
\end{equation}
whose restriction to $\Fl_{w}$ is the identity map for the constant sheaf $\QQ$. For each $y\in I_{*}, y\leq w$, the restriction of $\Phi_{w}$ induces an involution:
\begin{equation*}
\calH^{i}_{y}\Phi_{w}:\calH^{i}_{y}\bS_{w}=\tau^{*}\calH^{i}_{y}(\tau^{*}\bS_{w})\to\calH^{i}_{y}\bS_{w}
\end{equation*}
where the first equality comes from the definition of $\tau^{*}$. Then
\begin{equation}\label{defineP}
P^{\sigma}_{y,w}(q)=\sum_{i\in\ZZ}\tr(\calH^{i}_{y}\Phi_{w},\calH^{i}_{y}\bS_{w})q^{i/2}.
\end{equation}
It is known that $\calH^{i}_{y}\bS_{w}=0$ for odd $i$(see \cite[Theorem 4.2]{KL2} for the case $W$ finite, \cite[Theorem 5.5]{KL2} for the case $W$ affine; see also \cite[A.7]{Ga} for the affine case), therefore $P^{\sigma}_{y,w}\in\ZZ[q]$.

\subsection{Affine Grassmannian and the geometric Satake equivalence}\label{ss:Gr}
Let $\Gr=G((t))/G[[t]]$ be the affine Grassmannian of $G$, which is also a locally finite union of projective varieties of increasing dimensions. The left translation by $G[[t]]$ on $\Gr$ has orbits indexed by $W_{J}$-orbits on $\Lambda$. For each dominant coweight $\lambda\in\Lambda^{+}$, there is a unique $G[[t]]$-orbit $\Gr_{\lambda}$ containing $t^{\lambda}$ (which also contains $t^{\lambda'}$ for any $\lambda'$ in the same $W_{J}$-orbit of $\lambda$). The dimension of $\Gr_{\lambda}$ is $\jiao{2\rho,\lambda}$, where $2\rho$ is the sum of positive roots of $G$.
  
Let $\calS=P_{G[[t]]}(\Gr)$ be the category of $G[[t]]$-equivariant perverse sheaves on $\Gr$ which are supported on finitely many $G[[t]]$-orbits. This abelian category carries a convolution product $\odot:\calS\times\calS\to\calS$ (implicit in \cite{L83}, see \cite[Proposition 2.2.1]{G}), which is equipped with an obvious associativity constraint and a less obvious commutativity constraint (based on ideas of Drinfeld, see an exposition in \cite[Section 5]{MV}) making $(\calS,\odot)$ a tensor category (the convolution product is usually denoted by $*$ in literature, and we change it to $\odot$ to avoid confusion with the involution $*$). Let $\Vect^{\gr}$ be the category of finite dimensional graded $\QQ$-vector spaces (the commutativity constraint is {\em not} adjusted by the Koszul sign convention, so $\Vect^{\gr}\cong\Rep(\Gm)$ as tensor categories). Consider the functor
\begin{eqnarray*}
\Hb&:&\calS\to\Vect^{\gr}\\
&&\bK\mapsto\bigoplus_{i\in\ZZ}\cohog{i}{\Gr,\bK}.
\end{eqnarray*}
This functor carries a tensor structure (see \cite[Proposition 3.4.1]{G} and \cite[Proposition 6.3]{MV}, note that the commutativity constraint of $\calS$ is adjusted by a sign in \cite[Paragraph after Remark 6.2]{MV} in order to make $\Hb$ a tensor functor). 

Composing $\Hb$ with the forgetful functor $\Vect^{\gr}\to\Vect$ (the category of finite dimensional vector spaces), we get a fiber functor $\upH$ of the tensor category $\calS$, hence an algebraic group $\dG=\Aut^{\otimes}(\upH)$ over $\QQ$. In \cite[Theorem 3.8.1]{G} (with the corrected commutativity constraint by Drinfeld and based on results of \cite{L83}), it is proved that $\dG$ is a connected split reductive group over $\QQ$ whose root datum is dual to $G$. The proof in \cite[Theorem 7.3]{MV} in fact equips $\dG$ with a maximal torus $\dT$ with a canonical identification $\xch(\dT)=\Lambda=\xcoch(T)$. In fact, the functor $\Hb$ factors as
\begin{equation*}
\Hb:\calS\xrightarrow{\oplus_{\lambda\in\Lambda}F_{\lambda}}\Vect^{\Lambda}\xrightarrow{\jiao{2\rho,-}}\Vect^{\gr}
\end{equation*}
Here the first arrow is the sum of {\em weight functors} introduced in \cite[Theorem 3.6]{MV}; the second functor turns a $\Lambda$-graded vector space $\oplus_{\lambda}V^{\lambda}$ into a $\ZZ$-graded one $V^{i}:=\oplus_{\jiao{2\rho,\lambda}=i}V^{\lambda}$.
Under the identification $\calS\isom\Rep(\dG)$, the functor $\Hb$ then factors as
\begin{equation*}
\Rep(\dG)\to\Rep(\dT)\to\Rep(\Gm)
\end{equation*}
induced by the homomorphisms $2\rho:\Gm\to\dT\hookrightarrow\dG$.

\subsection{Geometric interpretation of $P^{\s}_{d_{\mu},d_{\l}}(q)$}
For each $\lambda\in\Lambda^{+}$, let $\bC_{\lambda}$ be the shifted intersection cohomology complex of the closure $\Gr_{\leq\lambda}$ of $\Gr_{\lambda}$, such that $\bC_{\l}|_{\Gr_{\l}}=\QQ$. The involution $\tau$ of $G((t))$ again induces a functor
\begin{equation}\label{tauS}
\tau^{*}:\calS\to\calS.
\end{equation}
One can similarly define the stalks $\calH^{i}_{\mu}\bC_{\lambda}$ for $\mu\leq\lambda\in\Lambda^{+}$, which again vanishes for odd $i$. Each double coset $G[[t]]t^{\lambda}G[[t]]$ is sent to $G[[t]]t^{-\lambda^{*}}G[[t]]$. By the definition of $*$, we have $-\lambda^{*}=\leftexp{w_{J}}{\lambda})$, hence $G[[t]]t^{-\lambda^{*}}G[[t]]=G[[t]]t^{\lambda}G[[t]]$, i.e., each $G[[t]]$-double coset in $G((t))$ is stable under $\tau$ (this is equivalent to saying that the longest element in each $W_{J}$-double coset belongs to the set $I_{*}$ of $*$-twisted involutions). This means one can fix an isomorphism
\begin{equation}\label{invC}
\Psi_{\lambda}:\tau^{*}\bC_{\lambda}\isom\bC_{\lambda}
\end{equation} 
which is the identity when restricted to $\Gr_{\lambda}$. This isomorphism similarly induces an involution:
\begin{equation*}
\calH^{i}_{\mu}\Psi_{\lambda}:\calH^{i}_{\mu}\bC_{\lambda}=\calH^{i}_{\mu}(\tau^{*}\bC_{\lambda})\to\calH^{i}_{\mu}\bC_{\lambda}.
\end{equation*} 

We have a projection map $\pi:\Fl\to\Gr$. For each $\lambda\in\Lambda^{+}$, the preimage $\pi^{-1}(\Gr_{\leq\lambda})=\Fl_{\leq d_{\lambda}}$ (recall $d_{\lambda}\in W_{J}\lambda W_{J}$ is the longest element). Since $\Fl_{\leq d_{\lambda}}\to\Gr_{\leq\lambda}$ is smooth, we have an isomorphism $\phi_{\lambda}:\pi^{*}\bC_{\lambda}\cong\bS_{d_{\lambda}}$, which can be made canonical by requiring its restriction to $\Fl_{d_{\lambda}}$ to be the identity map on the constant sheaf. Moreover, the isomorphism $\phi_{\lambda}$ clearly intertwines $\Psi_{\lambda}$ and $\Phi_{d_{\lambda}}$. Using $\phi_{\lambda}$, we get a commutative diagram
\begin{equation*}
\xymatrix{\calH^{j}_{\mu}\bC_{\lambda}\ar[rr]^{\calH^{j}_{\mu}\phi_{\lambda}}\ar[d]_{\calH^{j}_{\mu}\Psi_{\lambda}} && \calH^{j}_{d_{\mu}}\bS_{d_{\lambda}}\ar[d]^{\calH^{j}_{\mu}\Phi_{\lambda}}\\
\calH^{j}_{\mu}\bC_{\lambda}\ar[rr]^{\calH^{j}_{\mu}\phi_{\lambda}} && \calH^{j}_{d_{\mu}}\bS_{d_{\lambda}}}
\end{equation*}
Therefore, from \eqref{defineP} we get
\begin{equation}\label{Grdef}
P^{\sigma}_{d_{\mu},d_{\lambda}}(q)=\sum_{j\in\ZZ}\tr(\calH^{2j}_{\mu}\Psi_{\lambda},\calH^{2j}_{\mu}\bC_{\lambda})q^{j}.
\end{equation}

\subsection{Loop group of a compact form} At certain points in the proof of the main theorem, it is convenient to take an alternative point of view of the affine Grassmannian $\Gr$, namely the space of polynomial loops on the compact form of $G$. We remark that the switch of viewpoint is not necessary for the proof, but it makes the idea of the proof more transparent.
 
Let $K\subset G(\CC)$ be a compact real form which is stable under $*$ (for example, we may define $K$ using the Cartan involution $\dotw_{J}*$, for any lifting of $\dotw_{J}$ of $w_{J}$ to $N_{G}(T)$). Let $\Omega=\Omega_{\pol}K$ be the space of polynomial loops on $K$ based at the identity element $1\in K$ (see \cite[\S3.5]{PS}). By \cite[Theorem 8.6.3]{PS}, there is a homeomorphism
\begin{equation*}
\iota:\Omega\stackrel{\wt\iota}{\hookrightarrow} G(\CC((t)))\xrightarrow{p}\Gr(\CC).
\end{equation*} 
The stratification of $\Gr$ by $\{\Gr_{\lambda}\}_{\lambda\in\Lambda^{+}}$ gives a Whitney stratification of $\Omega$.  We denote the strata by $\Omega_{\lambda}$ with closure $\Omega_{\leq\lambda}$. Let $D^{b}(\Omega)=\varinjlim_{\lambda}D^{b}(\Omega_{\leq\lambda})$. Let $\calS_{K}$ be the full subcategory of $D^{b}(\Omega)$ consisting of perverse sheaves which are locally constant along each strata $\Omega_{\lambda}$. 

Let $m_{K}:\Omega\times\Omega\to\Omega$ be the multiplication map. This is stratified in the sense that $m_{K}(\Omega_{\leq\lambda}\times\Omega_{\leq\mu})=\Omega_{\leq\lambda+\mu}$ for $\lambda,\mu\in\Lambda^{+}$. Define\begin{eqnarray*}
\odot_{K}&:&D^{b}(\Omega)\times D^{b}(\Omega)\to D^{b}(\Omega)\\
&&(\bK_{1},\bK_{2})\mapsto m_{K!}(\bK_{1}\boxtimes\bK_{2}).
\end{eqnarray*}
Let 
\begin{equation*}
\Hb:\calS_{K}\to\Vect^{\gr}
\end{equation*}
be the functor of taking total cohomology.

The involution $\tau:k\mapsto (k^{*})^{-1}$ on $K$ induces an involution $\tau_{K}$ on $\Omega$, which gives the pullback functor
\begin{equation*}
\tau^{*}_{K}:D^{b}(\Omega)\to D^{b}(\Omega).
\end{equation*}

\begin{lemma}\label{l:GtoK}
\begin{enumerate}
\item []
\item The functor $\odot_{K}$ has image in $\calS_{K}$, and there is a natural associativity constraint making $(\calS_{K},\odot_{K})$ a monoidal category; $\Hb:\calS_{K}\to\Vect^{\gr}$ is naturally a monoidal functor.

\item\label{equiv} The pull-back functor $\iota^{*}$ gives a monoidal equivalence $\iota^{*}:(\calS,\odot)\to(\calS_{K},\odot_{K})$.

\item\label{Heq} There is a natural isomorphism of monoidal functors $\Hb\circ\iota^{*}\cong\Hb:\calS\to\Vect^{\gr}$. 

\item\label{tau} The functor $\tau^{*}_{K}$ sends $\calS_{K}$ to $\calS_{K}$; $\tau^{*}$ and $\tau^{*}_{K}$ are naturally intertwined under $\iota^{*}$. 

\end{enumerate}
\end{lemma}
\begin{proof}
(1)(2) The functor $\iota^{*}$ identifies $\calS_{K}$ with the category of perverse sheaves on $\Gr$ locally constant along the strata $\Gr_{\lambda}$. By \cite[Proposition A.1]{MV} the latter category is canonically equivalent to $\calS$. To prove (1) and (2), it suffices to give $\iota^{*}$ a monoidal structure. Recall that the convolution product $\odot$ on $\calS$ is defined as
\begin{equation*}
\bK_{1}\odot\bK_{2}=m_{!}(\bK_{1}\boxdot\bK_{2})
\end{equation*}
Here $m:G((t))\twtimes{G[[t]]}\Gr\to\Gr$ is the multiplication map, $\bK_{1}\boxdot\bK_{2}$ is the perverse sheaf on $G((t))\twtimes{G[[t]]}\Gr$ characterized by 
\begin{equation}\label{boxdot}
p'^{*}\bK_{1}\boxdot\bK_{2}=p^{*}\bK_{1}\boxtimes\bK_{2} \textup{ on } G((t))\times\Gr,
\end{equation}
where $p:G((t))\to\Gr, p':G((t))\times\Gr\to G((t))\twtimes{G[[t]]}\Gr$ are the projections. To give $\iota^{*}$ a tensor structure, we need to give a canonical isomorphism
\begin{equation*}
m_{K!}(\iota^{*}\bK_{1}\boxtimes\iota^{*}\bK_{2})\cong \iota^{*}m_{!}(\bK_{1}\boxdot\bK_{2})
\end{equation*}
for any $\bK_{1},\bK_{2}\in\calS$. Note that we have commutative diagram
\begin{equation}\label{mm}
\xymatrix{\Omega\times\Omega\ar[r]^{\iota_{2}}\ar[d]_{m_{K}} & G((t))\twtimes{G[[t]]}\Gr\ar[d]^{m}\\
\Omega\ar[r]^{\iota} & \Gr}
\end{equation}
where $\iota_{2}$ is given by the composition
\begin{equation*}
\Omega\times\Omega\xrightarrow{\wt{\iota}\times\iota} G((t))\times\Gr\xrightarrow{p'} G((t))\twtimes{G[[t]]}\Gr.
\end{equation*}
It is easy to see that $\iota_{2}$ is also a homeomorphism, so \eqref{mm} is a Cartesian diagram. Therefore we have a canonical isomorphism
\begin{eqnarray*}
&&\iota^{*}m_{!}(\bK_{1}\boxdot\bK_{2})\cong m_{K!}\iota^{*}_{2}(\bK_{1}\boxdot\bK_{2})\\
&=&m_{K!}(\wt{\iota}\times\iota)^{*}p'^{*}(\bK_{1}\boxdot\bK_{2})\stackrel{\eqref{boxdot}}{=}m_{K!}(\wt{\iota}\times\iota)^{*}(p^{*}\bK_{1}\boxtimes\bK_{2})\\
&=&m_{K!}(\wt{\iota}^{*}p^{*}\bK_{1}\boxtimes\iota^{*}\bK_{2})=m_{K!}(\iota^{*}\bK_{1}\boxtimes\iota^{*}\bK_{2})
\end{eqnarray*}
It is easy to check these isomorphisms are compatible with the associativity constraints.

(3) is obvious.

(4) For each $\bK\in\calS$, we need to give a functorial isomorphism
\begin{equation*}
\iota^{*}\tau^{*}\bK\isom\tau^{*}_{K}\iota^{*}\bK.
\end{equation*}
Recall $\iota$ factors as $\Omega\xrightarrow{\wt{\iota}}G((t))\xrightarrow{p}\Gr$ and $\wt{\iota}\tau_{K}=\tau\wt{\iota}$, where $\tau:g\mapsto (g^{*})^{-1}$ is the anti-automorphism of $G((t))$. Therefore
\begin{equation*}
\iota^{*}\tau^{*}\bK=\wt{\iota}^{*}p^{*}\tau^{*}\bK=\wt{\iota}^{*}\tau^{*}p^{*}\bK=\tau^{*}_{K}\wt{\iota}^{*}p^{*}\bK=\tau^{*}_{K}\iota^{*}\bK.
\end{equation*}
This gives the desired isomorphism.
\end{proof}

Using part \eqref{equiv} of Lemma \ref{l:GtoK}, one can transfer the commutativity constraint of $(\calS,\odot)$ to $(\calS_{K},\odot_{K})$ making the latter a tensor category. Part \eqref{Heq} of Lemma \ref{l:GtoK} then gives the functor $\Hb$ a tensor (in addition to monoidal) structure.

% proof

\section{Proof of Theorem \ref{th:main}}\label{s:proof}
For a monoidal category $(\calC,\otimes)$, we let $(\calC,\otimes^{\sigma})$ be the same category equipped with a new functor
$\otimes^{\sigma}:\calC\times\calC\to\calC$ given by $X\otimes^{\sigma}Y:=Y\otimes X$. It is easy to check that $(\calC, \otimes^{\sigma})$ also carries a monoidal structure.

\begin{lemma}
The functor $\tau^{*}:\calS\to\calS$ carries a natural structure of a monoidal functor
\begin{equation*}
\tau^{*}:(\calS,\odot)\to(\calS,\odot^{\sigma}).
\end{equation*}
\end{lemma}
\begin{proof}
Using Lemma \ref{l:GtoK}\eqref{equiv} and \eqref{tau}, it suffices to construct the monoidal structure of $\tau^{*}_{K}$. Let $\sigma:\Omega\times\Omega\to\Omega\times\Omega$ be the involution which interchanges two factors. Since $\tau_{K}$ is an anti-involution, we have a Cartesian diagram
\begin{equation}\label{Cart}
\xymatrix{\Omega\times\Omega\ar[r]^{\tau_{K}\times\tau_{K}}\ar[d]^{m_{K}\circ\sigma} & \Omega\times\Omega\ar[d]^{m_{K}}\\ \Omega\ar[r]^{\tau_{K}} & \Omega}
\end{equation}
Therefore by proper base change, for any $\bK_{1},\bK_{2}\in\calS_{K}$, we have a canonical isomorphism
\begin{equation*}
\tau^{*}_{K}m_{K!}(\bK_{1}\boxtimes\bK_{2})\cong(m_{K}\circ\sigma)_{!}(\tau^{*}_{K}\bK_{1}\boxtimes\tau^{*}_{K}\bK_{2})=m_{K!}(\tau^{*}_{K}\bK_{2}\boxtimes\tau^{*}_{K}\bK_{2}).
\end{equation*}
By the definition of $\odot_{K}$, we get a canonical isomorphism
\begin{equation*}
\tau^{*}_{K}(\bK_{1}\odot_{K}\bK_{2})\isom\tau^{*}_{K}\bK_{2}\odot_{K}\tau^{*}_{K}\bK_{1}.
\end{equation*}
It is easy to check that these isomorphisms are compatible with the associativity constraint and the unit objects of $(\calS_{K},\odot_{K})$ and $(\calS_{K},\odot_{K}^{\sigma})$. This finishes the proof of the lemma.
\end{proof}

Let $\Hbs:(\calS,\odot^{\sigma})\to(\Vect^{\gr},\otimes)$ be the same functor as $\Hb$, except that we change its monoidal structure to the one of $\Hb$ composed with the commutativity constraint of $\otimes$ for $\Vect^{\gr}$, so that $\Hbs$ is also a tensor functor.

\begin{lemma}\label{l:coho} There is a natural isomorphism $\gamma:\Hbs\circ\tau^{*}\isom\Hb$, which preserves the monoidal structures of both functors.
\end{lemma}
\begin{proof} Using Lemma \ref{l:GtoK}, it suffices to give a natural isomorphism $\gamma_{K}:\Hb\circ\tau^{*}_{K}\isom\Hb$ between functors $\calS_{K}\to\Vect^{\gr}$, which preserves the monoidal structures. Since $\tau_{K}$ is an automorphism of $\Omega$, we have a canonical isomorphism $\Hb(\Omega,\tau_{K}^{*}\bK)\isom\Hb(\Omega,\bK)$, which gives the desired $\gamma_{K}$. It remains to check that $\gamma$ preserves the monoidal structures. But this is also obvious from the natural monoidal structure of $\Hb:\calS_{K}\to\Vect^{\gr}$.
\end{proof}

Suppose we have two Tannakian categories $(\calC,\otimes)$ and $(\calD,\otimes)$ equipped with fiber functors $\omega_{\calC}$ and $\omega_{\calD}$ into $\Vect_{k}$ respectively ($k$ is a field). Let $F:(\calC,\otimes)\to(\calD,\otimes)$ be a monoidal functor equipped with a monoidal isomorphism $\phi:\omega_{\calD}\circ F\isom\omega_{\calC}$. Then $\phi$ induces a homomorphism of algebraic groups over $k$:
\begin{eqnarray*}
(F,\phi)^{\#}&:&\Aut^{\otimes}(\omega_{\calD})\to\Aut^{\otimes}(\omega_{\calC})\\
&&(\omega_{\calD}\xrightarrow{h}\omega_{\calD})\mapsto(\omega_{\calC}\xrightarrow{\phi^{-1}}\omega_{\calD}\circ F\xrightarrow{h\circ\id_{F}}\omega_{\calD}\circ F\xrightarrow{\phi}\omega_{\calC}).
\end{eqnarray*}
Note that the tensor morphisms between tensor functors only uses their structures as monoidal functors, therefore the above definition makes sense even if $F$ is only a monoidal functor. More generally, if $\omega_{\calC}$ and $\omega_{\calD}$ take values in another Tannakian category $\calV$ equipped with a fiber functor $\omega:\calV\to\Vect$, then $F$ induces a homomorphism of algebraic groups $(F,\phi)^{\#}:\Aut^{\otimes}(\omega\circ\omega_{\calD})\to\Aut^{\otimes}(\omega\circ\omega_{\calC})$ making the following diagram commutative
\begin{equation*}
\xymatrix{& \Aut^{\otimes}(\omega) \ar[dl]_{\omega_{\calD}^{\#}}\ar[dr]^{\omega_{\calC}^{\#}}\\
\Aut^{\otimes}(\omega\circ\omega_{\calD})\ar[rr]^{(F,\phi)^{\#}} & & \Aut^{\otimes}(\omega\circ\omega_{\calC})}
\end{equation*}

We apply the above remarks to the situation
\begin{equation*}
\xymatrix{(\calS,\odot)\ar[rr]^{\tau^{*}}\ar[dr]_{\Hb} & \ar@{}[d]|{\stackrel{\gamma}{\Leftarrow}} & (\calS,\odot^{\sigma})\ar[dl]^{\Hbs}\\
& \Vect^{\gr}}
\end{equation*}
and get a commutative diagram of algebraic groups over $\QQ$:
\begin{equation*}
\xymatrix{& \Gm \ar[dl]_{2\rho}\ar[dr]^{2\rho}\\
\dG\ar[rr]^{(\tau^{*},\gamma)^{\#}} & & \dG}
\end{equation*}
In other words, $(\tau^{*},\gamma)^{\#}$ is an automorphism of $\dG$ commuting with elements in the torus $2\rho(\Gm)$. Since $\tau^{*}$ does not change the isomorphism classes of irreducible objects in $\calS$, this automorphism must be inner. Therefore $(\tau^{*},\gamma)^{\#}$ determines an element $g\in\dT$ (note that $\dG$ is of adjoint form). 

Using the commutative constraint of $(\calS,\odot)$, the identity functor gives a monoidal equivalence
\begin{equation*}
\id^{\sigma}_{\calS}:(\calS,\odot)\isom(\calS,\odot^{\sigma}).
\end{equation*}
There is a unique natural isomorphism of monoidal functors $\Theta:\tau^{*}\isom\id^{\sigma}_{\calS}$ making
\begin{equation*}
\id_{\Hbs}\circ\Theta=\gamma:\Hbs\circ\tau^{*}\to\Hbs\circ\id^{\sigma}_{\calS}=\Hb.
\end{equation*}
In fact, identifying $\calS$ with $\Rep(\dG)$, the functor $\tau^{*}$ sends $V\in\Rep(\dG)$ (with the action $\alpha:\dG\to\Aut(V)$) to the same vector space $V$ with the new action $\dG\xrightarrow{\Ad(g)}\dG\xrightarrow{\alpha}\Aut(V)$. Then the effect of the natural isomorphism $\Theta$ on $V$ is given by $\alpha(g^{-1}):V\to V$.

\begin{lemma}\label{l:g}
\begin{enumerate}
\item []
\item The element $g\in\dT(\QQ)$ is $(-1)^{\rho}$, the image of $-1$ under the cocharacter $\rho:\Gm\to\dT$ (note that $\dG$ is of adjoint type, so $\rho$ is a cocharacter of $\dT$). 
\item The effect of the natural isomorphism $\Theta$ on the intersection complex $\bC_{\lambda}[\jiao{2\rho,\l}]\in\calS$ is $(-1)^{\jiao{\rho,\lambda}}\Psi_{\lambda}$.
\item The action of the involution $\tau^{*}_{K}$ on $\IH^{2j}(\Omega_{\leq\l})$ is by $(-1)^{j}$.  
\end{enumerate} 
\end{lemma}

% check shift of C_{\l} in the proof.
\begin{proof}
Let $\lambda\in\Lambda^{+}$. The action of $g^{-1}$ on $\Hb(\Omega,\bC_{\lambda})[\jiao{2\rho,\l}]=\IHb(\Omega_{\leq\lambda})[\jiao{2\rho,\l}]=V_{\lambda}\in\Rep(\dG)$ is given by the composition
\begin{equation*}
\IHb(\Omega_{\leq\lambda})\xrightarrow{\tau^{*}_{K}}\IHb(\Omega_{\leq\lambda})=\Hb(\Omega,\tau^{*}_{K}\bC_{\lambda})\xrightarrow{\Hb(\Omega,\Theta_{\lambda})}\Hb(\Omega,\bC_{\lambda})=\IHb(\Omega_{\leq\lambda}).
\end{equation*}
where the first arrow is the pull-back along the anti-involution $\tau_{K}$ of $\Omega_{\leq\lambda}$ and $\Theta_{\lambda}:\tau^{*}_{K}\bC_{\l}\to\bC_{\l}$ is induced from the effect of $\Theta$ on $\bC_{\lambda}[\jiao{2\rho,\l}]\in\calS$. Since the only automorphisms of $\bC_{\lambda}$ are scalars, the isomorphisms $\Theta_{\lambda}$ and $\Psi_{\lambda}$ must be related by $\Theta_{\lambda}=c_{\lambda}\Psi_{\lambda}$ for some constant $c_{\lambda}\in\QQ^{\times}$: the restriction of $\Theta_{\lambda}$ on $\Omega_{\lambda}$ is given by multiplication by $c_{\lambda}$ on the constant sheaf. 

The stratum $\Omega_{\lambda}$ homotopy retracts to the $K$-orbit of $t^{\lambda}$, which is a partial flag variety $G/P_{\lambda}=K/P_{\lambda}\cap K$ (see \cite[Top of page 100]{MV}). The action of $\tau_{K}$ on $\Ad(K)t^{\lambda}\cong K/P_{\lambda}\cap K$ is given by
\begin{equation*}
kt^{\lambda}k^{-1}\mapsto (k^{*}t^{\lambda^{*}}k^{*,-1})^{-1}=k^{*}t^{-\lambda^{*}}k^{*,-1}=k^{*}\dotw_{J}t^{\lambda}\dotw^{-1}_{0}k^{*,-1}.
\end{equation*}
Therefore the induced action of $\tau^{*}_{K}$ on $K/P_{\lambda}\cap K$ is given by $k\mod P_{\lambda}\cap K\mapsto k^{*}w_{J}\mod P_{\lambda}\cap K$ (any lifting $\dotw_{J}\in N_{T\cap K}(K)$ normalizes $P_{\lambda}\cap K$, hence the right translation makes sense).

Let $j_{\lambda}:\Omega_{\lambda}\hookrightarrow\Omega_{\leq\lambda}$ be the inclusion. We have a commutative diagram
\begin{equation}\label{restr}
\xymatrix{\IH^{i}(\Omega_{\leq\lambda})\ar[d]^{c_{\lambda}^{-1}g^{-1}}\ar[r]^{j^{*}_{\lambda}} & \cohog{i}{\Omega_{\lambda}}\ar[r]^{\sim}\ar[d]^{\tau_{K}^{*}} & \cohog{i}{K/P_{\lambda}\cap K}\ar[d]^{w_{J}*}\\
\IH^{i}(\Omega_{\leq\lambda})\ar[r]^{j^{*}_{\lambda}}& \cohog{i}{\Omega_{\lambda}}\ar[r]^{\sim} & \cohog{i}{K/P_{\lambda}\cap K}}
\end{equation}
When $i\leq2$, the horizontal restriction maps are isomorphisms. In fact, from the stratification $\Omega_{\leq\lambda}$ by the open $\Omega_{\lambda}$ and the closed complement $z:\Omega_{<\lambda}\hookrightarrow\Omega_{\leq\lambda}$, we get an exact sequence
\begin{equation}\label{exact}
\cohog{i}{\Omega_{<\lambda},z^{!}\bC_{\lambda}}\to\IH^{i}(\Omega_{\leq\lambda})\to\cohog{i}{\Omega_{\lambda}}\to\cohog{i}{\Omega_{<\lambda},z^{!}\bC_{\lambda}}
\end{equation}
Since $\dim\Omega_{<\lambda}\leq\jiao{2\rho,\lambda}-2$ and $z^{!}\bC_{\lambda}[\jiao{2\rho,\l}]$ lies in perverse degree $\geq1$, $z^{!}\bC_{\l}$ lies in the usual cohomological degree $\geq3$. This implies $\cohog{i}{\Omega_{<\lambda},z^{!}\bC_{\lambda}}=0$ for $i\leq 2$ hence the isomorphism follows from the exact sequence \eqref{exact}.

We claim that the action $\tau^{*}_{K}: k\mapsto k^{*}w_{J}$ on the partial flag variety $K/P_{\lambda}\cap K$ induces $-1$ on $\cohog{2}{K/P_{\lambda}\cap K}$. In fact, $\cohog{2}{K/P_{\lambda}\cap K,\QQ}\hookrightarrow\cohog{2}{K/T,\QQ}\cong\xch(T)_{\QQ}$ by pull-back along the projection $K/T\cap K\to K/P_{\lambda}\cap K$, and this map is equivariant under the $(W\rtimes\Out(G))_{\lambda}$-actions (subscript $\lambda$ means stabilizer of $\lambda$ under the $W\rtimes\Out(G)$-action on $\Lambda=\xcoch(T)$). Since $*w_{J}=w_{J}*\in W\rtimes\Out(G)$ acts on $\Lambda$ by $-1$ by definition, the claim follows.

Since $\tau^{*}_{K}$ induces the identity action on $\cohog{0}{K/P_{\lambda}\cap K}$, $c_{\lambda}^{-1}g^{-1}$ acts by identity on $\IH^{0}(\Omega_{\leq\lambda})$ by diagram \eqref{restr}. Since $\tau^{*}_{K}$ acts by -1 on $\cohog{2}{K/P_{\lambda}\cap K}$ by the above claim, $c_{\lambda}^{-1}g^{-1}$ acts on $\IH^{2}(\Omega_{\leq\lambda})$ by multiplication by $-1$ by diagram \eqref{restr}. 

Recall that the grading on $\IHb(\Omega_{\leq\lambda})[\jiao{2\rho,\l}]\cong V_{\l}$ comes from the action of the cocharacter $2\rho:\Gm\to\dT$ on $V_{\lambda}$. Let $V_{\lambda}(\mu)$ be the weight space of weight $\mu$ under the $\dT$-action, we have
\begin{equation*}
\IH^{i}(\Omega_{\leq\lambda})=\bigoplus_{\jiao{2\rho,\mu}=i}V_{\lambda}(\mu).
\end{equation*}
In particular,
\begin{eqnarray*}
\IH^{0}(\Omega_{\leq\lambda})&=&V_{\lambda}(\leftexp{w_{J}}\lambda);\\
\IH^{2}(\Omega_{\leq\lambda})&=&\bigoplus V_{\lambda}(\leftexp{w_{J}}{\lambda}+\alpha^{\vee}_{i})
\end{eqnarray*}
where the sum over simple roots $\alpha^{\vee}_{i}$ of $\dG$. Therefore, the previous paragraph implies
\begin{eqnarray}
\label{g0}c_{\lambda}^{-1}g^{-\leftexp{w_{J}}{\lambda}}&=&1;\\
\notag c_{\lambda}^{-1}g^{-\leftexp{w_{J}}{\lambda}-\alpha^{\vee}_{i}}&=&-1 \textup{ for all simple roots }\alpha^{\vee}_{i}.
\end{eqnarray}
Comparing these two equations we conclude that $g^{\alpha^{\vee}_{i}}=-1$ for all simple roots $\alpha^{\vee}_{i}$ of $\dG$. On the other hand, $(-1)^{\rho}$ also has this property. Since $\dG$ is adjoint, an element in $\dT$ is determined by its image under simple roots, therefore $g=(-1)^{\rho}$. This proves (2). Plugging this back into \eqref{g0}, we conclude that $c_{\lambda}=((-1)^{\rho})^{-\leftexp{w_{J}}{\lambda}}=(-1)^{\jiao{\rho,-\leftexp{w_{J}}{\lambda}}}=(-1)^{\jiao{\rho,\lambda}}$. This proves (1). Now (3) follows easily from (1) and (2).
\end{proof}

\subsection{Completion of the proof} By \eqref{Grdef}, it suffices to show that $\calH^{2j}_{\mu}\Psi_{\lambda}$ acts on $\calH^{2j}_{\mu}\bC_{\lambda}$ by $(-1)^{j}$.

We extend the partially ordered set $\{\mu\in\Lambda^{+},\mu\leq\lambda\}$ into a totally ordered one, and denote the total ordering still by $\leq$. For any $\mu\leq\lambda$, let $\Omega_{[\mu,\lambda]}=\Omega_{\leq\lambda}-\Omega_{<\mu}$. Similarly $\Omega_{(\mu,\lambda]}=\Omega_{\leq\lambda}-\Omega_{\leq\mu}$. Then we have a long exact sequence
\begin{equation*}
\cdots\to\cohoc{i}{\Omega_{(\mu,\lambda]},\bC_{\lambda}}\to\cohoc{i}{\Omega_{[\mu,\lambda]},\bC_{\lambda}}\to\oplus_{a+b=i}\cohoc{a}{\Omega_{\mu}}\otimes\calH^{b}_{\mu}\bC_{\lambda}\to\cdots
\end{equation*}
Since $\Omega_{\mu}=\Gr_{\mu}$ is an affine space bundle over a partial flag variety $G/P_{\mu}$, we have that $\Hb_{c}(\Omega_{\mu})\cong\Hb(G/P_{\mu})[-\jiao{2\rho,\mu}+\dim G/P_{\mu}]$ which is concentrated in even degrees. We also know that $\calH^{b}_{\mu}\bC_{\lambda}$ vanishes for odd $b$. Therefore the third term in the above exact sequence vanishes for odd $i$. Using decreasing induction for $\mu$ (starting with $\lambda$), we conclude that each $\Hb_{c}(\Omega_{[\mu,\lambda]},\bC_{\lambda})$ is concentrated in even degrees, and the above long exact sequence becomes a short one for even $i$. This gives a canonical decreasing filtration
\begin{equation*}
F^{\geq\mu}\IHb(\Omega_{\leq\lambda}):=\Hb_{c}(\Omega_{[\mu,\lambda]},\bC_{\lambda})
\end{equation*}
with associated graded pieces
\begin{equation}\label{grF}
\gr_{F}^{\mu}\IHb(\Omega_{\leq\lambda})=\Hb_{c}(\Omega_{\mu})\otimes\calH^{\bullet}_{\mu}\bC_{\lambda}.
\end{equation}
The action of $\tau_{K}^{*}$ preserves each $F^{\geq\mu}$, and the induced action on the associated graded pieces takes the form
\begin{equation}\label{graction}
\gr^{\mu}_{F}\tau^{*}_{K}=(\tau_{K}|_{\Omega_{\mu}})^{*}\otimes\calH^{\bullet}_{\mu}\Psi_{\lambda}:\Hb_{c}(\Omega_{\mu})\otimes\calH^{\bullet}_{\mu}\bC_{\lambda}\to\Hb_{c}(\Omega_{\mu})\otimes\calH^{\bullet}_{\mu}\bC_{\lambda}.
\end{equation}
By Lemma \ref{l:g}(3), the action of $\tau^{*}_{K}$ on the top-dimensional cohomology $\cohoc{2\jiao{2\rho,\mu}}{\Omega_{\mu}}\cong\IH^{2\jiao{2\rho,\mu}}(\Omega_{\leq\mu})$ is via multiplication by $(-1)^{\jiao{2\rho,\mu}}=1$; the action of $\tau^{*}_{K}$ on $\cohoc{2\jiao{2\rho,\mu}}{\Omega_{\mu}}\otimes\calH^{2j}_{\mu}\bC_{\l}\subset\gr^{\mu}_{F}\IH^{2j+2\jiao{2\rho,\l}}(\Omega_{\leq\lambda})$, as a subquotient of $\IH^{2j+2\jiao{2\rho,\l}}(\Omega_{\leq\lambda})$, is via multiplication by $(-1)^{j+\jiao{2\rho,\l}}=(-1)^{j}$. Therefore, by \eqref{graction}, $\calH^{2j}_{\mu}\Psi_{\lambda}$ acts on $\calH^{2j}_{\mu}\bC_{\lambda}$ via multiplication by $(-1)^{j}$. This finishes the proof of Theorem \ref{th:main}.

\section{Geometric proof of Theorem \ref{th:Z-q}}\label{s:g2}

The proof of Theorem \ref{th:Z-q} will become transparent once we give the cohomological interpretation of the $Z^{\s}$-polynomials.

\subsection{Affine flag variety via a compact form} We already see that $\Omega=\Omega K\isom\Gr(\CC)$ is a homeomorphism. We need analogous statement for the affine flag variety. Let $T_{c}=K\cap T$ be the maximal torus in $K$. Then the inclusion $K\subset G(\CC)$ induces a homeomorphism $K/T_{c}\isom (G/B)(\CC)$. The multiplication $(g,kT_{c})\mapsto gk\bI$ gives a continuous map $\iota_{\Fl}:\Omega\times K/T_{c}\to\Fl(\CC)$ making the following diagram commutative
\begin{equation}\label{Omfl}
\xymatrix{\Omega\times K/T_{c}\ar[d]^{\pr_{\Omega}}\ar[r]^{\iota_{\Fl}} & \Fl(\CC)\ar[d]^{\pi}\\ \Omega\ar[r]^{\iota} & \Gr(\CC)}
\end{equation}
It is easy to check that $\iota_{\Fl}$ is bijective on points, hence a homeomorphism because it is a continuous map from a compact space to a Hausdorff one. Moreover, $\iota_{\Fl}$ is $T_{c}$-equivariant, where $T_{c}$ acts on $\Omega\times K/T_{c}$ diagonally by conjugation and left translation, and it acts on $\Fl(\CC)$ by left translation.

Let $\Xi=K/T_{c}\times\Omega\times K/T_{c}$, on which $K$ acts diagonally via left translation on $K/T_{c}$ and via conjugation on $\Omega$. The space $\Xi$ also admits an involution $\wt\tau:(k_{1}T_{c},g,k_{2}T_{c})\mapsto(k_{2}^{*}T_{c},(g^{*})^{-1}, k_{1}^{*}T_{c})$, which intertwines the original diagonal $K$-action and the that action pre-composed with $*$. We may rewrite $\left[T_{c}\backslash(\Omega\times K/T_{c})\right]$ as $[K\backslash\Xi]$, so that the homeomorphism $\iota_{\Fl}$ can be rewritten as a map of topological stacks
\begin{eqnarray}\label{maptoFl}
[K\backslash\Xi]&\xrightarrow{\wt\iota_{\Fl}}& T_{c}\backslash\Fl\to\bI\backslash\Fl\\
\notag(k_{1}T_{c},g,k_{2}T_{c})&\mapsto&  T_{c}k_{1}^{-1}gk_{2}\bI \mapsto\bI k_{1}^{-1}gk_{2}\bI.
\end{eqnarray} 
which intertwines the involutions $\wt\tau$ and $\tau$. Since $\iota_{\Fl}$ is a homeomorphism, so is $\wt\iota_{\Fl}$ (by which we mean that it comes from a $K$-equivariant homeomorphism of topological spaces). Via \eqref{maptoFl}, we may define $\Xi_{w}$ (resp. $\Xi_{\leq w}$) as the preimage of $\bI\backslash\Fl_{w}$ (resp. $\bI\backslash\Fl_{\leq w}$) for each $w\in W$. Then $\Xi_{w}$ is $K$-equivariantly homeomorphic to the twisted product $K\twtimes{T_{c}}\Fl_{w}$.

Recall from \eqref{invSw} we have an isomorphism $\Phi_{w}:\tau^{*}\bS_{w}\isom\bS_{w}$ in the category $D_{\bI}(\Fl)$ for $w\in I_{*}$, where $\bS_{w}$ is the shifted intersection cohomology sheaf of $\Fl_{\leq w}$. This induces an involution on $\bI$-equivariant cohomology
\begin{equation*}
\tau^{*}=\Hb_{\bI}(\Fl,\Phi_{w}):\IHb_{\bI}(\Fl_{\leq w})\isom\IHb_{\bI}(\Fl_{\leq w}).
\end{equation*}

\begin{lemma}\label{l:Zs} Let $r$ be the rank of $G$. Then
\begin{equation*}
\sum_{j\in\ZZ}\tr\left(\tau^{*},\IH^{2j}_{\bI}(\Fl_{\leq w})\right)q^{j}=q^{\ell(w)}(1-q)^{e(*)-r}(1+q)^{-e(*)}Z^{\s}_{w}(q^{-1})
\end{equation*}
as elements in $\ZZ[[q]]$. Here $e(*)$ is the dimension of the $(-1)$-eigenspace of $*:\L_{\QQ}\to\L_{\QQ}$.
\end{lemma}
\begin{proof}
By \eqref{maptoFl}, $\IHb_{\bI}(\Fl_{\leq w})\cong\IHb_{K}(\Xi_{\leq w})$. We think of $\bS_{w}$ as the intersection complex of $\Xi_{\leq w}$ which is the constant sheaf on $\Xi_{w}$. The stratification of $\Xi$ by $\Xi_{\leq w}$ gives a spectral sequence with the $E_{2}$-page consisting of $\Hb_{K}(\Xi_{y},i^{!}_{y}\bS_{w})$ abutting to $\IHb_{K}(\Xi_{\leq w})$. Here $i_{y}:\Xi_{y}\hookrightarrow\Xi$ is the inclusion. Since $i^{!}_{y}\bS_{w}$ is a sum of constant sheaves on $\Xi_{y}$ concentrated on even degrees, and $\Hb_{K}(\Xi_{y})\cong\Hb_{T}(\pt)$ is also concentrated in even degrees, the spectral sequence necessarily degenerates at $E_{2}$. Therefore $\IHb_{K}(\Xi_{\leq w})$ admits an increasing filtration indexed by $\{y\leq w\}$ with $\gr_{y}\IHb_{K}(\Xi_{\leq w})=\Hb_{K}(\Xi_{y},i^{!}_{y}\bS_{w})$. The involution $\wt\tau^{*}$ on $\IHb_{K}(\Xi_{\leq w})$ maps $\gr_{y}$ to $\gr_{(y^{*})^{-1}}$, therefore its trace is the sum of traces on $\gr_{y}$ for $y\in I_{*}$, i.e., 
\begin{eqnarray}\label{sumy}
\sum_{j}\tr(\wt\tau^{*},\IH^{2j}_{K}(\Xi_{\leq w}))q^{j}=\sum_{y\leq w,y\in I_{*}}\sum_{j\in\ZZ}\tr(\wt\tau^{*},\upH^{2j}_{K}(\Xi_{y},i^{!}_{y}\bS_{w}))q^{j}.
\end{eqnarray}
Verdier duality gives an isomorphism in $D_{K}(\Xi_{y})$ commuting with the involutions induced by $\Phi_{w}$
\begin{equation*}
i^{!}_{y}\bS_{w}\cong\bigoplus_{k}\calH^{2\ell(w)-2\ell(y)-2k}\bS_{w}[-2k].
\end{equation*}
Hence
\begin{eqnarray}\label{twoprod}
&&\sum_{j\in\ZZ}\tr(\wt\tau^{*},\upH^{2j}_{K}(\Xi_{y},i^{!}_{y}\bS_{w}))q^{j}\\
\notag &=&\sum_{k\in\ZZ}\tr(\calH^{2\ell(w)-2\ell(y)-2k}_{y}\Phi_{w}, \calH^{2\ell(w)-2\ell(y)-2k}_{y}\bS_{w})q^{k}\sum_{j\in\ZZ}\tr(\wt\tau^{*},\upH^{2k}_{K}(\Xi_{y}))q^{j}\\
\notag &=&q^{\ell(w)-\ell(y)}P^{\s}_{y,w}(q^{-1})\sum_{j\in\ZZ}\tr(\wt\tau^{*},\upH^{2j}_{K}(\Xi_{y}))q^{j}.
\end{eqnarray}
Since $[K\backslash\Xi_{y}]\cong[T_{c}\backslash\Fl_{y}]$, it has the same cohomology as $[T\backslash T\doty T/T]$, where $\doty$ is a lifting of $y$ to $G((t))$. Let $T_{y}=\{(\bary(t^{-1}),t),t\in T\}\subset T\times T$ ($\bary$ is the image of $y$ in $\barW$) be the stabilizer of the $T\times T$-action on $T\doty T$ via left and right translations. The involution $\tau$ on $[T\backslash T\doty T/T]\cong[\pt/T_{y}]$ is then induced by the involution $(\bary(t^{-1}),t)\mapsto((t^{*})^{-1},\bary(t^{*}))$ of $T_{y}$. We identify $T_{y}$ with $T$ via the second projection, then the involution of $[\pt/T_{y}]=[\pt/T]$ induced by $\tau^{*}$ comes from $t\mapsto \bary(t^{*})$. This involution gives a decomposition $\frt=\frt_{+}\oplus\frt_{-}$ of the Lie algebra $\frt\cong\L_{\CC}$ of $T$ into $(+1)$ and $(-1)$-eigenspaces, with dimensions $r-e(\bary*)$ and $e(\bary*)$ respectively (see remarks following \eqref{zeta} for notations). Since
\begin{equation*}
\Hb_{K}(\Xi_{y})\cong\Hb_{T_{y}}(\pt)\cong\Sym(\frt^{\vee}[-2])\cong\Sym(\frt_{+}^{\vee}[-2])\otimes\Sym(\frt_{-}^{\vee}[-2]),
\end{equation*}
therefore
\begin{equation*}
\sum_{j\in\ZZ}\tr(\tau^{*},\upH^{2j}_{K}(\Xi_{y}))q^{j}=\sum_{j\geq0}\dim\Sym^{j}(\frt^{\vee}_{+})q^{j}\sum_{k\geq0}\dim\Sym^{k}(\frt^{\vee}_{-})(-q)^{k}=(1-q)^{e(\bary*)-r}(1+q)^{-e(\bary*)}.
\end{equation*}
Plugging this into \eqref{twoprod} and then into \eqref{sumy}, we get
\begin{eqnarray*}
&&\sum_{j}\tr(\tau^{*},\IH^{2j}_{K}(\Xi_{\leq w}))q^{j}\\
&=&\sum_{y\leq w,y\in I_{*}}q^{\ell(w)-\ell(y)}P^{\s}_{y,w}(q^{-1})(1-q)^{e(\bary*)-r}(1+q)^{-e(\bary*)}\\
&=&q^{\ell(w)}(1-q)^{-r}\sum_{y\leq w,y\in I_{*}}P^{\s}_{y,w}(q^{-1})q^{-\ell(y)}\left(\frac{q^{-1}-1}{q^{-1}+1}\right)^{e(\bary*)}\\
&=&q^{\ell(w)}(1-q)^{-r}\left(\frac{q^{-1}-1}{q^{-1}+1}\right)^{e(*)}Z^{\s}_{w}(q^{-1})=q^{\ell(w)}(1-q)^{e(*)-r}(1+q)^{-e(*)}Z^{\s}_{w}(q^{-1}).
\end{eqnarray*} 
\end{proof}

In the situation of the affine Grassmannian, the isomorphism \eqref{invC} induces an involution on the global sections $\tau^{*}_{K}:\IHb(\Omega_{\leq\l})\isom\IHb(\Omega_{\leq\l})$.

% remember to normalize C_{\l}, IH() to start with degree 0 (so shifted perverse).

\begin{lemma}\label{l:tZ} For $\l\in\L^{+}$, we have
\begin{equation}\label{Poin}
\sum_{i\in\ZZ}\tr(\tau^{*}_{K},\IH^{2i}(\Omega_{\leq\l}))q^{i}=\tZ^{\s}_{d_{\l}}(q).
\end{equation}
\end{lemma}
\begin{proof}
The map \eqref{maptoFl} induces an isomorphism on intersection cohomology commuting with the relevant involutions:
\begin{equation}\label{3ten}
\IHb_{\bI}(\Fl_{\leq d_{\l}})\cong\IHb_{K}(\Xi_{\leq d_{\l}})\cong\IHb_{K}(\Omega_{\leq\l})\otimes_{\Hb_{K}(\pt)}\Hb_{K}(K/T_{c}\times K/T_{c}).
\end{equation}
where the last equality comes from the degeneration of the Leray spectral sequence at $E_{2}$ since all the relevant cohomology groups are concentrated in even degrees. Note that in \eqref{3ten}, the involution on $\Hb_{K}(K/T_{c}\times K/T_{c})$ is induced by $(k_{1}T_{c},k_{2}T_{c})\mapsto(k_{2}^{*}T_{c},k_{1}^{*}T_{c})$, and the involution on $\Hb_{K}(\pt)$ is induced by the involution $*$ of $K$.

Another spectral sequence argument shows that we have an isomorphism
\begin{equation*}
\IHb_{K}(\Omega_{\leq\l})\cong\Hb_{K}(\pt)\otimes\IHb(\Omega_{\leq\l})
\end{equation*}
commuting with the obvious involutions (the one on $\Hb_{K}(\pt)$ is again induced by $*$, and the ones involving $\Omega_{\leq\l}$ are given by $\tau^{*}_{K}$). Combining this with \eqref{3ten} we get an isomorphism
\begin{equation*}
\IHb_{\bI}(\Fl_{\leq d_{\l}})\isom\IHb(\Omega_{\leq\l})\otimes\Hb_{K}(K/T_{c}\times K/T_{c})
\end{equation*}
intertwining the involutions on both sides which we specified before. The special case $\l=0$, $d_{\l}=w_{J}$ gives $\IHb_{\bI}(\Fl_{\leq w_{J}})\cong\Hb_{K}(K/T_{c}\times K/T_{c})$. Therefore
\begin{equation*}
\IHb_{\bI}(\Fl_{\leq d_{\l}})\isom\IHb(\Omega_{\leq\l})\otimes\IHb_{\bI}(\Fl_{\leq w_{J}})
\end{equation*}
commuting with the relevant involutions. Taking the Poincar\'e polynomials with respect to the traces of these involutions, and using Lemma \ref{l:Zs}, we get
\begin{eqnarray*}
&&q^{\ell(d_{\l})}(1-q)^{e(*)-r}(1+q)^{-e(*)}Z^{\s}_{d_{\l}}(q^{-1})\\
&=&q^{\ell(w_{J})}(1-q)^{e(*)-r}(1+q)^{-e(*)}Z^{\s}_{w_{J}}(q^{-1})\sum_{j\in\ZZ}\tr(\tau^{*}_{K},\IH^{2j}(\Omega_{\leq\l}))q^{j}.
\end{eqnarray*}
In view of the definition of $\tZ^{\s}_{d_{\l}}(q)$ in \eqref{definetZ}, we get
\begin{equation*}
\sum_{j\in\ZZ}\tr(\tau^{*}_{K},\IH^{2j}(\Omega_{\leq\l}))q^{j}=q^{\ell(d_{\l})-\ell(w_{J})}\tZ^{\s}_{d_{\l}}(q^{-1}).
\end{equation*}
Let $Q_{\l}(q)$ denote the left side. Substituting $q^{-1}$ for $q$ in the above, we get
\begin{equation*}
Q_{\l}(q^{-1})=q^{-\ell(d_{\l})+\ell(w_{J})}\tZ^{\s}_{d_{\l}}(q).
\end{equation*}
Poincar\'e duality for $\IHb(\Omega_{\leq\l})$ (which has dimension $\jiao{2\rho,\l}$) implies $Q_{\l}(q)=q^{\jiao{2\rho,\l}}Q_{\l}(q^{-1})$. Therefore
\begin{equation*}
Q_{\l}(q)=q^{\jiao{2\rho,\l}}Q_{\l}(q^{-1})=q^{\jiao{2\rho,\l}+\ell(w_{J})-\ell(d_{\l})}\tZ^{\s}_{d_{\l}}(q).
\end{equation*}
Since $\ell(d_{\l})=\ell(w_{J})+\jiao{2\rho,\l}$ (i.e, $\dim\Fl_{\leq d_{\l}}=\dim G/B+\dim\Gr_{\leq\lambda}$), the above equality implies \eqref{Poin}.
\end{proof}

\subsection{Completion of the proof} By Lemma \ref{l:g}(3), the involution $\tau^{*}_{K}$ acts on $\IH^{2j}(\Omega_{\leq\l})$ via $(-1)^{j}$. Therefore, by Lemma \ref{l:tZ}, we have
\begin{equation}\label{tZsPoin}
\tZ^{\s}_{d_{\l}}(q)=\sum_{j\in\ZZ}(-1)^{j}\dim\IH^{2j}(\Omega_{\leq\l})q^{j}=\sum_{j\in\ZZ}\dim\IH^{2j}(\Omega_{\leq\l})(-q)^{j}.
\end{equation}
On the other hand, the argument using the filtration \eqref{grF} shows that $Z_{d_{\l}}(q)$ is the Poincar\'e polynomial for $\IHb(\Fl_{\leq d_{\l}})$:
\begin{equation*}
Z_{w}(q)=\sum_{j\in\ZZ}\dim\IH^{2j}(\Fl_{\leq w})q^{j}.
\end{equation*}
Using the homeomorphism $\iota_{\Fl}$ and the diagram \eqref{Omfl}, we have $\IHb(\Fl_{\leq d_{\l}})\cong\IHb(\Omega_{\leq\l})\otimes\Hb(K/T_{c})\cong\IHb(\Omega_{\l})\otimes\IHb(\Fl_{\leq w_{J}})$. Therefore $Z_{d_{\l}}(q)$ is the product of $Z_{w_{J}}(q)$ with the Poincar\'e polynomial of $\IHb(\Omega_{\leq\l})$. By the definition of $\tZ_{d_{\l}}(q)$ in \eqref{definetZ}, we have
\begin{equation}\label{tZPoin}
\tZ_{d_{\l}}(q)=Z_{d_{\l}}(q)Z_{w_{J}}(q)^{-1}=\sum_{j\in\ZZ}\dim\IH^{2j}(\Omega_{\leq\l})q^{j}.
\end{equation}
The theorem now follows by comparing \eqref{tZsPoin} and \eqref{tZPoin}.

\section{Algebraic Proof of Theorem \ref{th:Z-q}}\label{s:a2}
Now we start the algebraic proof of Theorem \ref{th:Z-q}. Using \cite[3.6(f)]{L} and Theorem \ref{th:main} we see that
\begin{eqnarray*}
\tZ^\s_{d_\l}(q)&=&\sum_{\mu\in\L^+;d_{\mu}\leq d_\l}P_{d_{\mu},d_\l}^\s(q)\zeta\left(\sum_{y\in W_{J}\mu W_{J};y\in I_*}a_y\right)Z^\s_{w_{J}}(q)^{-1}\\
&=& \sum_{\mu\in\L^+;d_{\mu}\leq d_\l}P_{d_{\mu},d_\l}(-q)\zeta\left(\sum_{y\in W_{J}\mu W_{J};y\in I_*}a_y\right)Z^\s_{w_{J}}(q)^{-1}.
\end{eqnarray*}
On the other hand 
\begin{equation*}
\tZ_{d_\l}(-q)=\sum_{\mu\in\L^+;d_{\mu}\leq d_\l}
P_{d_{\mu},d_\l}(-q)\sum_{y\in W_{J}\mu W_{J}}(-q)^{\ell(y)}Z_{w_{J}}(-q)^{-1}.
\end{equation*}
Hence to prove Theorem \ref{th:Z-q} it is enough to show that for any double coset $W_{J}\mu W_{J}$ we have
\begin{equation}\label{5.3a}
\zeta\left(\sum_{y\in W_{J}\mu W_{J}\cap I_*}a_y\right)Z^\s_{w_{J}}(q)^{-1}=\sum_{y\in W_{J}\mu W_{J}}(-q)^{\ell(y)}Z_{w_{J}}(-q)^{-1}.
\end{equation}

We fixed such a double coset $W_{J}\mu W_{J}$ for the rest of this section, where $\mu\in\L^{+}\cap W_{J}\mu W_{J}$ is the unique dominant translation. Let $d=d_{\mu}$ (resp. $b$) be the element of maximal (resp. minimal) length in $W_{J}\mu W_{J}$. 

We shall be interested also in some parabolic analogues of $Z_w(q),Z^\s_w(q)$. For any $H\subsetneqq S$ let $W_H$ be the subgroup of $W$ generated by $H$ so that $(W_H,H)$ is a finite Coxeter group;
let $w_H$ be the longest element of $W_H$. We also set $\bP_H=\sum_{x\in W_H}q^{\ell(x)}\in\NN[q]$ so that $Z_{w_H}(q)=\bP_H(q)$. Recall that $J=S-\{s_{0}\}$, and our previous notation $W_{J}, w_{J}$ is consistent with the new notation. 

If in addition we are given an involution $\t:W_H\to W_H$ leaving $H$ stable, we set (as in 
\cite[5.1]{L}) $\bP_{H,\t}=\sum_{x\in W_H;\t(x)=x}q^{\ell(x)}\in\NN[q]$. By
\cite[5.9]{L} we have $Z^\s_{w_{J}}(q)=\bP_J(q^2)\bP_{J,*}(q)^{-1}$ (we use also that $P_{y,w_{J}}^\s(q)=1$ for any $y\in W_{J}$, see \cite[3.6(f)]{L}). 

Let $H=J\cap bJb^{-1}$. Let $\ep:W_{H^*}\to W_{H^*}$ be the involution $y\mapsto b^{-1}y^*b$ ($H^{*}$ is the image of $H$ under $*$). From \cite[5.10]{L} we have
\begin{equation*}
\zeta(\sum_{y\in W_{J}\mu W_{J}\cap I_*}a_y)=\zeta(a_b)\bP_J(q^2)\bP_{H^*,\ep}(q)^{-1}.
\end{equation*}
Similarly,
\begin{equation*}
\sum_{y\in W_{J}\mu W_{J}}q^{\ell(y)}=q^{\ell(b)}\bP_J(q^{2})\bP_{H^*}(q)^{-1}.
\end{equation*}
We see that \eqref{5.3a} is equivalent to the following statement:
\begin{equation}\label{5.3b}
\zeta(a_b)\bP_{J,*}(q)\bP_{H^*,\ep}(q)^{-1}=(-q)^{\ell(b)}\bP_J(-q)\bP_{H^*}(-q)^{-1}.
\end{equation}

\begin{lemma}\label{5.4a} The involution $\ep$ on $W_{H^{*}}$ is the same as $\Ad(w_{H^*})$; i.e., $b^{-1}y^*b=w_{H^*}yw_{H^*}$ for all $y\in W_{H^{*}}$.
\end{lemma}
\begin{proof}
We shall denote the inverse of $\mu$ by $\mu^{-1}$ instead of $-\mu$ as before. We have $\mu w_{J}=d=w_{J}w_Hbw_{J}$. Hence $\mu=w_{J}w_Hb$. Now $W_J\times W_J$ acts transitively on $W_{J}\mu W_{J}$ by left and right multiplication and the isotropy group of $\mu$ is isomorphic to $W_{\mu}:=\{w\in W_{J};\leftexp{w}\mu=\mu\}$. Hence $|W_{J}\mu W_{J}|=|W_J|^2/|W_{\mu}|$. By \cite[1.1]{L} we have also 
$|W_{J}\mu W_{J}|=|W_J|^2/|W_H|$ hence $|W_{\mu}|=|W_H|$. We show that $W_{\mu}\subset W_{H^*}$. We have $W_H=W_{J}\cap bW_{J}b^{-1}$; applying $*$ we deduce $W_{H^*}=W_{J}\cap b^{-1} W_{J}b$. Hence it is enough to show that $W_{\mu}\subset b^{-1} W_{J}b$. Since $\mu=w_{J}w_Hb$ we have $\mu^{-1}W_{J}\mu=b^{-1} w_Hw_{J} W_{J}w_{J}w_Hb=b^{-1} W_{J}b$. If $w\in W_{\mu}$ then $w\mu=\mu w$ hence $\mu w\mu^{-1}=w\in W_{J}$; thus $W_{\mu}\subset \mu^{-1}W_{J}\mu=b^{-1} W_{J}b$. We have shown that $W_{\mu}\subset W_{H^*}$. Since the last two groups have the same order we see that $W_{\mu}=W_{H^*}$. Hence to prove (a) it is enough to show that for any $y\in W_{\mu}$ we have 
$b^{-1} y^*b=w_{H^*}yw_{H^*}$ that is (after applying $*$) $byb^{-1}=w_Hy^*w_H$. Since $b=w_Hw_{J}\mu$, it is enough to show that for $y\in W_{\mu}$ we have $w_{J}\mu y \mu^{-1}w_{J}=y^*$, or, using $\mu y=y\mu$, that $w_{J}yw_{J}=y^*$. This follows from the definition of $*$ in Section \ref{ss:aff}. This proves the lemma.
\end{proof}

\begin{lemma}\label{5.5a} If $L\subsetneqq S$ and $\Ad(w_{L})$ is the conjugation by $w_{L}$ on $W_{L}$, then
\begin{equation*}
\bP_{L,\Ad(w_{L})}(q)=\bP_L(-q)\left(\frac{1+q}{1-q}\right)^{n_L},
\end{equation*}
where $n_L$ is the number of odd exponents of $W_L$.
\end{lemma}
\begin{proof}
Let $e_i (i\in X)$ be the exponents of $W_L$. We have $X=X'\sqcup X''$ where 
$X'=\{i\in X;e_i\text{ is odd}\}$, $X''=\{i\in X;e_i \text{ is even}\}$. It is well known that 
\begin{equation*}
\bP_L(q)=\prod_{i\in X}\frac{q^{e_i+1}-1}{q-1}.
\end{equation*}
It follows that
\begin{equation}\label{5.5b}
\bP_L(-q)=\prod_{i\in X'}\frac{q^{e_i+1}-1}{-q-1}\prod_{i\in X''}\frac{q^{e_i+1}+1}{q+1}
\end{equation}
We have
\begin{equation}\label{5.5c}
\bP_{L,\Ad(w_{L})}(q)=\prod_{i\in X'}\frac{q^{e_i+1}-1}{q-1}\prod_{i\in X''}\frac{q^{e_i+1}+1}{q+1}.
\end{equation}
Here, the left hand side evaluated at a prime power $q$ is the number of $\FF_q$-rational Borel subgroups of 
a semisimple algebraic group defined over $\FF_q$ which is twisted according to the opposition involution. This can be computed from the known formula for the number of rational points of such an algebraic groups given in 
\cite[\S11]{St}. Now the lemma follows from \eqref{5.5b} and \eqref{5.5c}.
\end{proof}

\subsection{}\label{ss:5.6}
Using Lemma \ref{5.4a} and \ref{5.5a} (applied to $L=H^{*}$) and the definition of $\zeta$ we see that the desired equality \eqref{5.3b} is equivalent to
\begin{equation*}
q^{\ell(b)}\left(\frac{q-1}{q+1}\right)^{\phi(b)}\left(\frac{1+q}{1-q}\right)^{n_J-n_{H^*}}=(-q)^{\ell(b)},
\end{equation*}
that is, to the equality
\begin{equation}\label{5.6a}
\phi(b)=n_J-n_{H^*}.
\end{equation}   
Here we use that $\phi(w)=\ell(w)(\mod2)$ for any $w\in I_*$, see \cite[4.5]{L}.

Define $\phi':\{z\in W_{H^*};\ep(z)=z^{-1}\}\to\NN$ in terms of $(W_{H^*},\ep)$ in the same way as $\phi$ was defined in terms of $W$ and $*$ in \cite[4.5]{L} (using the difference of the dimension of the $(-1)$-eigenspaces of $w\in W_{H^{*}}$ and $ww_{H^{*}}$ on the reflection representation of $W_{H^{*}}$). We show:
 
\begin{lemma}\label{l:5.6b} For any $z\in W_{H^*}$ such that $\ep(z)=z^{-1}$, we have $\phi(bz)=\phi'(z)+\phi(b)$.
\end{lemma}
\begin{proof}
We argue by induction on $\ell(z)$. If $z=1$ the result is clear. Now assume that $z\neq1$. We can find $s\in H^*$ such that $\ell(sz)<\ell(z)$. Assume first that $sz\neq z\ep(s)$.  Then $\ell(sz\ep(s))=\ell(z)-2$ hence by the induction
hypothesis we have $\phi(bsz\ep(s))=\phi'(sz\ep(s))+\phi(b)$. By definition, $\phi'(sz\ep(s))=\phi'(z)$. We have
$bsz\ep(s)=bsb^{-1} bz\ep(s)=\ep(s)^*bz\ep(s)$ and hence, by definition, $\phi(bsz\ep(s))=\phi(\ep(s)^*bz\ep(s))=\phi(bz)$. Thus $\phi(bz)=\phi'(z)+\phi(b)$.
%$$\ell(\ep(s)^*bz\ep(s))=\ell(bsz\ep(s))=\ell(b)+\ell(sz\ep(s))=\ell(b)+\ell(z)-2=\ell(bz)-2$$
Next we assume 
that $sz=z\ep(s)$.  Then $\ell(sz\ep(s))=\ell(z)-1$ hence by the induction hypothesis we have $\phi(bsz\ep(s))=\phi'(sz\ep(s))+\phi(b)$. By definition, $\phi'(sz\ep(s))=\phi'(z)-1$ and $\phi(bsz\ep(s))=\phi(\ep(s)^*bz\ep(s))=\phi(bz)-1$. 
%We have
%\begin{equation*}
%\ell(\ep(s)^*bz\ep(s))=\ell(bsz\ep(s))=\ell(b)+\ell(sz\ep(s))=\ell(b)+\ell(z)-1=\ell(bz)-1.
%\end{equation*}
Thus $\phi(bz)=\phi'(z)+\phi(b)$. This completes the proof of the lemma.
\end{proof}

\subsection{Completion of the proof} From Lemma \ref{l:5.6b} we deduce 
\begin{equation}\label{5.6c}
\phi(bw_{H^*})=\phi'(w_{H^*})+\phi(b).
\end{equation}
We have $d=cbw_{H^*}c^{*-1}$ where $c=w_{J}w_H$ (see \cite[\S1.2]{L}). 
%$\ell(d)=\ell(c)+\ell(b)+\ell(w_{H^*})+\ell(c^*)$  
From the definition of $\phi$ we see that $\phi(d)=\phi(bw_{H^*})$ hence, using \eqref{5.6c}, we have
\begin{equation*}\label{5.6d}
\phi(d)=\phi'(w_{H^*})+\phi(b).
\end{equation*}
Hence \eqref{5.6a} is equivalent to
\begin{equation}\label{5.6e}
\phi(d)-\phi'(w_{H^*})=n_J-n_{H^*}.
\end{equation}

For any linear map $A:\L_\QQ\to\L_\QQ$ (where $\L_{\QQ}=\L\otimes_{\ZZ}\QQ$), recall $e(A)$ is the dimension of the $(-1)$-eigenspace of $A$. We claim that
\begin{equation}\label{5.7a}
\phi(d)=e(w_{J}).
\end{equation}
In fact, if $w\in I_*$ with image $\barw\in\barW$, we have $\phi(w)=e(\barw *)-e(*)$. Since the action of $*$ is given by $x\mapsto-w_{J}(x)$), we have 
$\phi(w)=e(-\barw w_{J})-e(-w_{J})$. If $w=d$ then $d=tw_{J}$ ($t$ is the dominant translation) hence $\barw=w_{J}\in\barW\cong W_{J}$ and $\phi(d)=e(-\id)-e(-w_{J})$, which is equal to $e(w_{J})$. This proves \eqref{5.7a}.

Now let $R'$ be the reflection representation of $W_{H^*}$. For any linear map $A:R'\to R'$ we denote by $e'(A)$ the dimension of the $(-1)$-eigenspace of $A$. We claim that
\begin{equation}\label{5.7b}
\phi'(w_{H^*})=e'(w_{H^*}).
\end{equation}
In fact, from the definition we have $\phi'(w_{H^*})=e'(w_{H^*}\ep)-e'(\ep)$. Note that both $w_{H^*}$ and $\ep$ act naturally on $R'$; the action of $\ep$ is given by $x\mapsto-w_{H^*}x$ by Lemma \ref{5.4a}. Thus we have $\phi'(w_{H^*})=e'(-\id)-e'(-w_{H^*})=e'(w_{H^*})$. This proves \eqref{5.7b}.

Using \eqref{5.7a} and \eqref{5.7b} we see that the desired equality \eqref{5.6e} is equivalent to
\begin{equation}\label{5.7c}
e(w_{J})-e'(w_{H^*})=n_J-n_{H^*}.
\end{equation}
Now for any finite Weyl group, the dimension of the $(-1)$-eigenspace of the longest element acting on the reflection representation is equal to the number of odd exponents of that Weyl group, as one easily verifies. It follows that $e(w_{J})=n_J$, $e'(w_{H^*})=n_{H^*}$. Thus \eqref{5.7c} is proved. This completes the proof of Theorem \ref{th:Z-q}.

\quash{
\subsection{}\label{ss:5.8}
Let $\dG$ be a simple adjoint group over $\CC$ of type dual to that of $(W,S)$. We assume that we are
given a maximal torus $\dT$ of $\dG$, a Borel subgroup $\dB$ of $\dG$ containing $\dT$, an identification of
$\L\cong\xch(\dT)$ and an identification of $W_{J}$ with the Weyl group of $\dG$ with respect to
$\dT$ in such a way that the natural action of the last Weyl group on $\xch(\dT)$ becomes
the natural $W_{J}$ action on $\L$. Now any $\l\in\L^+$ can be viewed as the highest weight of a finite
dimensional irreducible $\dG$-module $V_\l$. 
}

\subsection{Signature of a hermitian form}\label{ss:sign}
Let $\dG$ be the Langlands dual of $G$ as before, with dual Cartan and Borel $\dT\subset\dB$. We identify the Weyl group of $\dG$ with $W_{J}$. Let $\l\in\L^+$, viewed as a dominant weight of $\dG$, and let $V_{\l}$ be the corresponding irreducible representation of $\dG$ with highest weight $\l$. In \cite{L97} a hermitian form $h_{\l}$ on $V_\l$ is constructed in terms of a semisimple element 
$s\in\dT$ with $s^{2}=1$. Here we shall take $s=(-1)^{\rho}$. The hermitian form $h_{\l}$ is invariant under
a real form of $\dG$ which can be shown to be quasi-split (for our choice of $s$) and admits a compact Cartan subgroup. Moreover, by \cite[2.9]{L97}, the signature of $h_{\l}$ is given by
\begin{equation}\label{5.9a}
\text{Signature}(h_{\l})=(-1)^{\jiao{\rho,\l}}\tr((-1)^{\rho},V_\l).
\end{equation}

Recall the following results from \cite{L83}. First, it is shown in \cite[6.1]{L83} that the multiplicity of the weight $\mu$ in $V_\l$ is equal to $P_{d_\mu,d_\l}(1)$. Second, we have the formula (see \cite[(8.10)]{L83} and its proof)
\begin{equation}\label{5.8c}
\tZ_{d_\l}(q)=q^{\jiao{\rho,\l}}\sum_{\mu\in\L^+;d_\mu\leq d_\l}P_{d_\mu,d_\l}(1)
\sum_{\mu\in W_{J}\mu}q^{\jiao{\rho,\mu}}.
\end{equation}
Setting $q=1$ we obtain that $\tZ_{d_{\l}}(1)=\dim V_{\l}$. Setting $q=-1$ in \eqref{5.8c}, we obtain
\begin{equation}\label{5.8d}
\tZ_{d_\l}(-1)=(-1)^{\jiao{\rho,\l}}\tr((-1)^{\rho},V_\l).
\end{equation}
We may also obtain \eqref{5.8d} from Lemma \ref{l:g}(3) and Lemma \ref{l:tZ}.
Combining \eqref{5.8d} with Theorem \ref{th:Z-q} and \eqref{5.9a}, we obtain
\begin{equation}\label{5.9d}
\text{Signature}(h_{\l})=\tZ^\s_{d_\l}(1).
\end{equation}
Thus, while $\tZ_{d_\l}(q)$ is a $q$-analogue of the dimension of $V_\l$, $\tZ^\s_{d_\l}(q)=\tZ_{d_\l}(-q)$ is the $q$-analogue of the signature of the hermitian form $h_{\l}$ on $V_\l$.

\begin{remark}
We expect that the hermitian form $h_{\l}$ on $V_\l$ is the complexification of the sum of the polarization Hodge structures $\IH^{2p}(\Gr_{\leq\l})$ (which only has $(p,p)$-classes). By the Riemann-Hodge bilinear relation, this pairing is positive (resp. negative) definite on $\IH^{2p}(\Gr_{\leq\l})$ when $p$ is even (resp. odd). Therefore the signature on the total intersection cohomology $\IHb(\Gr_{\leq\l})$ (which is also the signature of the Poincar\'e duality pairing) is also  calculated by $\tZ_{d_{\l}}(-1)=\tZ^{\s}_{d_{\l}}(1)$. 

%The phenomenon noted in Section \ref{ss:sign} is then related to the following fact: if $\IHb(X)$ has Hodge numbers $h^{p,q}$ are zero for $p\neq q$, then the intersection Euler characteristic of $X$ is obtained from the Poincar\'e polynomial $\sum_p h^{p,p}u^p$ by setting $u=1$, while the signature of the polarization on $\IHb(X)$ is obtained from the Poincar\'e polynomial by setting $u=-1$ (Hodge signature theorem).
\end{remark}

\section{Generalization}\label{s:gen}

\subsection{More involutions in affine Weyl groups}
In Section \ref{ss:aff}, we fixed a hyperspecial vertex $s_{0}\in S$ in the Dynkin diagram of $(W,S)$. Let $A=\Aut(W,S)$. Then $A$ has a subgroup
\begin{equation*}
A_{\Lambda}:=\{a\in\Aut(W,S)|\textup{ there exists }w\in W_{J}\textup { such that }a(\lambda)=\leftexp{w}{\lambda}\textup{ for all }\lambda\in\Lambda\}.
\end{equation*}
One may identify $A_{\Lambda}$ with the affine automorphisms fixing the standard alcove corresponding to $S$. It is easy to see that $A_{\Lambda}$ is normal in $A$. Let $\barA:=A/A_{\Lambda}$. The stabilizer of $s_{0}$ under $A$ is $A_{J}=\Aut(W_{J}, J)$, which projects isomorphically to $\barA$.

We recall the extended affine Weyl group is the semi-direct product $\tilW=W\rtimes A_{\Lambda}$, and it fits into an exact sequence
\begin{equation*}
1\to\tilLam\to\tilW\to\barW\to1
\end{equation*}
where $\tilLam$ is a lattice containing $\Lambda$ such that the projection $\tilLam\hookrightarrow\tilW\to A_{\Lambda}$ induces an isomorphism $\tilLam/\Lambda\cong A_{\Lambda}$.

\begin{lemma} Recall we have an involution $*\in A_{J}$ defined in \eqref{star}.
\begin{enumerate}
\item Every element in the coset $A_{\Lambda}*=*A_{\Lambda}\subset A$ is an involution.
\item For any hyperspecial vertex $s_{1}\in S$, there is a unique $a\in A_{\Lambda}*$ which sends $s_{0}$ to $s_{1}$.
\end{enumerate}
\end{lemma}
\begin{proof}
(1) The group $A_{J}$ acts on $A_{\L}$ by conjugation. This action can be seen explicitly as follows: $W_{J}\rtimes A_{J}$ acts on $\tilLam$ by the reflection action stabilizing $\L$. The action of $A_{J}$ on the quotient $A_{\L}=\tilLam/\L$ is then induced from this reflection action. In particular, the action of $*\in A_{J}$ on $\tilLam$ is via $\l\mapsto-\leftexp{w_{J}}\l$, which is congruent to $-\l$ modulo $\L$. Therefore $*$ acts on $A_{\L}$ by inversion, hence every element $a*\in A_{\L}*$ satisfies $(a*)^{2}=a(*a*)=aa^{-1}=1$. 

(2) It is well-known that $A_{\Lambda}$ permutes the hyperspecial vertices simply transitively. Then for any $a\in A_{\Lambda}$, we have $(a*)(s_{0})=a(s_{0})$ which exhaust all hyperspecial vertices exactly once as $a$ runs over $A_{\Lambda}$. 
\end{proof}

Let $s_{1}$ be another hyperspecial vertex in $S$. Let $\diamond\in A_{\Lambda}*$ be the unique involution taking $s_{0}$ to $s_{1}$, hence taking $J$ to $J^{\diamond}=S-\{s_{1}\}$. Let $I_{\diamond}=\{w\in W|w^{\diamond}=w^{-1}\}$ be the $\diamond$-twisted involutions in $W$. To avoid complicated subscripts, we denote $W_{J^{\diamond}}$ by $W_{J}^{\diamond}$ instead. 

The following theorem generalizes Theorem \ref{th:main}.
\begin{theorem}\label{th:gen}
\begin{enumerate}
\item []
\item Each double coset $W_{J}\backslash W/W_{J}^{\diamond}$ in $W$ is stable under the anti-involution $w\mapsto (w^{\diamond})^{-1}$. In particular, the longest element in each $(W_{J},W_{J}^{\diamond})$-double coset belongs to $I_{\diamond}$.
\item For longest representatives $d_{1}$ and $d_{2}$ of $(W_{J},W_{J}^{\diamond})$-double cosets in $W$, we have
\begin{equation*}
P_{d_1,d_2}^{\sigma,\diamond}(q)=P_{d_1,d_2}(-q).
\end{equation*}
Here the polynomials $P^{\sigma,\diamond}_{y,w}(q)$ ($y,w\in I_{\diamond}$) are the ones defined in \cite{L} in terms of $(W,S,\diamond)$.
\end{enumerate}
\end{theorem}

\subsection{Sketch of proof} We only indicate how to modify the proof of Theorem \ref{th:main} to give the proof of this theorem. 

The anti-involution $w\mapsto (w^{*})^{-1}$ extends to an anti-involution on $\tilW$ by the same formula \eqref{star} (except that $\lambda$ now is any element in $\tilLam$). Again each double coset $W_{J}\backslash\tilW/W_{J}$ is stable under this anti-involution. Write $\diamond=a*$ for $a\in A_{\Lambda}$, then $W_{J}^{\diamond}=a(W_{J})$. Multiplication by $a$ on the right gives a bijection
\begin{equation*}
W_{J}\backslash W/W_{J}^{\diamond}\leftrightarrow W_{J}\backslash W\cdot a/W_{J}\subset W_{J}\backslash\tilW/W_{J}.
\end{equation*}
This shows part (1) of Theorem \ref{th:gen}.

In the situation of Section \ref{ss:flag}, $G$ is a simply-connected group. Let $\Gad$ be the adjoint form of $G$, with maximal torus $\Tad=T/Z(G)$. Then we have a natural isomorphism $\tilLam\cong\xcoch(\Tad)$. The connected components of the affine Grassmannian $\Grad$ for $\Gad$ are indexed by $\tilLam/\Lambda$. The $\Gad[[t]]$-orbits on $\Grad$ are indexed by $\tilLam/W_{J}$, and the natural projection $\tilLam/W_{J}\to\tilLam/\Lambda$ indicates which orbit belongs to which connected component. Identifying $A_{\Lambda}$ with $\tilLam/\Lambda$, we denote the corresponding component of $\Grad$ by $\Grad_{a}$ ($a\in A_{\L}$ such that $\diamond=a*$). We may similarly define the Satake category $\calSad$ for $\Gad$ with simple objects $\bC_{\lambda}[\jiao{2\rho,\l}]$, $\lambda\in\tilLam^{+}$ (dominant coweights of $\Gad$). Via the fiber functor $\Hb$, $\calSad$ is equivalent to $\Rep(\dGsc)$, where $\dGsc$ is the simply-connected form of $\dG$. The same anti-involution $\tau^{*}$ defines a functor $(\calSad,\odot)\to(\calSad,\odot^{\sigma})$, and there is an isomorphism $\Psi_{\lambda}:\tau^{*}\bC_{\lambda}\isom\bC_{\lambda}$ normalized to be the identity on $\Grad_{\lambda}$, which induces an involution $\calH^{i}_{\mu}\Psi_{\lambda}$ on the stalks $\calH^{i}_{\mu}\bC_{\lambda}$ for $\mu\leq\lambda\in\tilLam^{+}$. Note that in the partial ordering of $\tilLam$, two elements are comparable only if they are congruent modulo $\L$.

Let $\dota\in N_{\Gad((t))}(\Tad((t)))$ be a lifting of $a\in A_{\Lambda}<\tilW$, then $\dota\Gad[[t]]\dota^{-1}$ is a hyperspecial parahoric subgroup of $\Gad((t))$ corresponding to the vertex $s_{1}=\diamond(s_{0})$. Let $\bP\subset G((t))$ be the hyperspecial parahoric subgroup (containing $\bI$) corresponding to $s_{1}$. Right multiplication by $\dota$ induces an isomorphism
\begin{equation}\label{Pa}
G((t))/\bP\isom\Gad((t))/\dota\Gad[[t]]\dota^{-1}\isom\Grad_{a}
\end{equation} 
which is equivariant under the left actions by $G[[t]]$. The double coset $G[[t]]\backslash G((t))/\bP$ is in bijection with $W_{J}\backslash W/W_{J}^{\diamond}$. As in \eqref{Grdef}, the coefficients of the polynomials $P^{\sigma,\diamond}_{d_{1},d_{2}}(q)$ are expressible as the traces of an involution on the stalks of the intersection cohomology complexes on $G[[t]]$-orbits of $G((t))/\bP$. Under the isomorphism \eqref{Pa}, we have the following formula generalizing \eqref{Grdef}:
\begin{equation*}
P^{\sigma,\diamond}_{d_{1},d_{2}}(q)=\sum_{j\in\ZZ}\tr(\calH^{2j}_{\mu}\Psi_{\lambda},\calH^{2j}_{\mu}\bC_{\lambda})q^{j}.
\end{equation*}
Here $\mu\leq\lambda\in\tilLam^{+}$ have image equal to $a$ in $\tilLam/\Lambda$, and $d_{1}$ (resp. $d_{2}$) is the longest element in the double coset $W_{J}\mu a^{-1}W_{J}^{\diamond}$ (resp. $W_{J}\lambda a^{-1}W_{J}^{\diamond}$).

So in order to prove Theorem \ref{th:gen}(2), it suffices to show that $\calH^{2j}_{\mu}\Psi_{\lambda}$ acts on $\calH^{2j}_{\mu}\bC_{\lambda}$ via multiplication by $(-1)^{j}$ for any $\mu\leq\lambda\in\tilLam^{+}$. The argument in Section \ref{s:proof} works up to Lemma \ref{l:coho}. The pair $(\tau^{*},\gamma)$ again determines the element $g=(-1)^{\rho}\in\dT<\Aut(\dGsc)$. However, a monoidal isomorphism $\Theta:\tau^{*}\isom\id^{\sigma}_{\calSad}$ is the same as the choice of an element $\tg\in\dTsc$ lifting $(-1)^{\rho}$: the effect of $\Theta$ on $V\in\Rep(\dGsc)\cong\calSad$ is the action of $\tg^{-1}$. Lemma \ref{l:g}(2) should say that the effect of $\Theta$ (or $\tg^{-1}$) on $\bC_{\lambda}[\jiao{2\rho,\l}]$ is $\tg^{-\leftexp{w_{J}}{\lambda}}\tau^{*}_{K}$. In the rest of the argument, we use \eqref{graction}. The piece $\cohoc{2\jiao{2\rho,\mu}}{\Omega_{\mu}}\otimes\calH^{2j}_{\mu}\bC_{\lambda}$ appears in degree $2\jiao{2\rho,\mu}+2j-\jiao{2\rho,\l}$ in $\IHb(\Omega_{\leq\l})[\jiao{2\rho,\l}]\cong V_{\l}$, hence it appears as a subquotient of $\oplus_{\nu}V_{\lambda}(\nu)$, where $\nu\in\tilLam$ has the same image as $\lambda$ and $\mu$ in $\tilLam/\Lambda$ and
\begin{equation}\label{nu}
\jiao{2\rho,\nu}=2\jiao{2\rho,\mu}+2j-\jiao{2\rho,\l}, \textup{ or }j=\jiao{\rho,\nu+\l-2\mu}.
\end{equation}
We write $\cohoc{2\jiao{2\rho,\mu}}{\Omega_{\mu}}\otimes\calH^{2j}_{\mu}\bC_{\lambda}=\oplus_{\nu}(\cohoc{2\jiao{2\rho,\mu}}{\Omega_{\mu}}\otimes\calH^{j}_{\mu}\bC_{\lambda})_{\nu}$ according to the weight decomposition. Therefore $\tg^{-1}$ or $\tg^{-\leftexp{w_{J}}{\lambda}}\tau^{*}_K$ acts on $(\cohoc{2\jiao{2\rho,\mu}}{\Omega_{\mu}}\otimes\calH^{2j}_{\mu}\bC_{\lambda})_{\nu}$ by $\tg^{-\nu}$. Specializing to $\lambda=\mu=\nu$, $\tg^{-\leftexp{w_{J}}{\mu}}\tau^{*}_{K}$ acts on $\cohoc{2\jiao{2\rho,\mu}}{\Omega_{\mu}}=\IH^{2\jiao{2\rho,\mu}}(\Omega_{\leq\mu})$ by $\tg^{-\mu}$. Therefore, by \eqref{graction}, the action of $\calH^{j}_{\mu}\Psi_{\lambda}$ on $\calH^{j}_{\mu}\bC_{\lambda}$ is given by
\begin{equation}\label{tg}
\tg^{-\nu+\leftexp{w_{J}}{\lambda}}\cdot(\tg^{-\mu+\leftexp{w_{J}}{\mu}})^{-1}=\tg^{-\nu+\mu+\leftexp{w_{J}}{(\lambda-\mu)}}.
\end{equation} 
Since $-\nu+\mu\in\Lambda$, we have $\tg^{-\nu+\mu}=g^{-\nu+\mu}=(-1)^{\jiao{\rho,-\nu+\mu}}$. Since $\lambda-\mu\in\Lambda$, we also have $\tg^{\leftexp{w_{J}}{(\lambda-\mu)}}=g^{\leftexp{w_{J}}{(\lambda-\mu)}}=(-1)^{\jiao{\rho,\leftexp{w_{J}}{(\lambda-\mu)}}}=(-1)^{\jiao{-\rho,\lambda-\mu}}$. Taking these two facts together we conclude that the expression \eqref{tg} is equal to
\begin{equation*}
(-1)^{\jiao{\rho,-\nu+\mu}}(-1)^{\jiao{-\rho,\lambda-\mu}}=(-1)^{\jiao{\rho,-\nu-\lambda+2\mu}},
\end{equation*}
which is equal to $(-1)^{j}$ by \eqref{nu}. This finishes the proof of Theorem \ref{th:gen}.

\begin{remark}
The results of Section \ref{s:g2} and \ref{s:a2} can also be extended to the setup in Section \ref{s:gen}. Thus $\dG$ can be replaced by the corresponding simply connected group whose irreducible finite dimensional representations carry a natural hermitian form as in \cite{L97} with signature expressible in terms analogous to \eqref{5.9d}. We omit the details.
\end{remark}

%%% reference %%%

\end{document}